\documentclass[oneside, 10pt, a4paper]{article} % use larger type; default would be 10pt

\usepackage[utf8]{inputenc} % set input encoding (not needed with XeLaTeX)

\usepackage{geometry} % to change the page dimensions
\usepackage{graphicx} % support the \includegraphics command and options

\usepackage[all]{xy}
\usepackage{amsmath}
\usepackage{amsthm}
\usepackage[shortlabels, inline]{enumitem}
\usepackage{mathrsfs}
\usepackage{amsfonts}
\usepackage{amssymb}
\usepackage{tikz}
\usepackage{tikz-cd} 
\usepackage{url}
\usepackage{mathtools}
\usepackage{rotating}
\usepackage[thinc]{esdiff}
\usepackage[english]{isodate}
\usepackage{csquotes}
\usepackage[most]{tcolorbox}
\tcbset{colback=yellow!10!white, colframe=red!50!black, 
        highlight math style= {enhanced, %<-- needed for the ’remember’ options
            colframe=red,colback=red!10!white,boxsep=0pt}
        }
\usetikzlibrary{matrix}
\usetikzlibrary{babel}
\usepackage{cancel}
\usepackage{booktabs} % for much better looking tables
\usepackage{array} % for better arrays (eg matrices) in maths
\usepackage{verbatim} % adds environment for commenting out blocks of text & for better verbatim
\usepackage{subfig} % make it possible to include more than one captioned figure/table in a single float
\usepackage{hyperref} % makes referenced sections/chapters/theorems etc into hyperlinks 
\hypersetup{
    colorlinks,
    citecolor=black,
    filecolor=black,
    linkcolor=black,
    urlcolor=black
}

\usepackage{fancyhdr} % This should be set AFTER setting up the page geometry
\usepackage{sectsty}
\allsectionsfont{\sffamily\mdseries\upshape} % (See the fntguide.pdf for font help)
\usepackage[nottoc,notlof,notlot]{tocbibind} % Put the bibliography in the ToC
\usepackage[titles,subfigure]{tocloft} % Alter the style of the Table of Contents

\makeatletter
\DeclareRobustCommand{\tvdots}{%
  \vbox{\baselineskip4\p@\lineskiplimit\z@\kern0\p@\hbox{.}\hbox{.}\hbox{.}}}
\makeatother

\theoremstyle{plain}
\newtheorem{thm}{Theorem}[section]
\newtheorem{prop}[thm]{Proposition}
\newtheorem{corr}[thm]{Corollary}
\newtheorem{lem}[thm]{Lemma}
\theoremstyle{definition}
\newtheorem{defn}[thm]{Definition}

\theoremstyle{remark}

\newtheorem*{rem}{Remark}

\newcommand{\map}[3]{#1\colon #2 \to #3}

\newcommand{\mb}{\mathbb}

\newcommand{\del}{\delta}
\newcommand{\sig}{\sigma}
\newcommand{\lam}{\lambda}
\newcommand{\Lam}{\Lambda}

\newcommand{\pr}{\prime}

\newcommand{\mc}{\mathcal}

\newcommand{\im}[1]{\mbox{im}(#1)}

\newcommand{\Hom}[3]{\mbox{Hom}_{#1}(#2, #3)}

\newcommand{\teha}[1]{\mbox{T}(EH(#1))}

\newcommand{\bsym}[1]{\boldsymbol{#1}}

\title{Spectral sequences for the cyclic cohomology of differential graded algebras}
\author{Andrew Phimister}
\date{} % Activate to display a given date or no date (if empty),
\begin{document}
\maketitle

\begin{abstract}
We construct a number of new spectral sequences for calculating the cyclic cohomology $HC^*_{dg}(A)$ of a differential graded algebra (dga). With these spectral sequences we prove some results about the low dimensional cyclic cohomology and demonstrate the existence of various maps between $HC^*_{dg}(A)$ and $HH^*_{dg}(A)$. We also briefly introduce variations on Hochschild and cyclic cohomology of a dga, namely the $n$-th partial Hochschild cohomology and $n$-partial cyclic cohomology. Finally, we show how these results can be extended naturally to the dg-category setting. In particular we define the Hochschild and cyclic cohomolgy of dg-categories and show that the spectral sequences we have constructed can be used in this setting as well.
\end{abstract}

\paragraph{\textbf{Note}} This preprint is a modified version of Chapters 3, 4.5 and 7.2 of the PhD thesis of the author~\cite{phimphd2025}.

\section{Introduction}

\paragraph{\textbf{Conventions}} $\bsym{1}_C$ will denote the identity map of $C$, although we will occasionally drop the subscript. $\Bbbk$ is a field of characteristic 0. We write $\mc{T}(C^*)$ to denote the total complex of a cochain complex $C^*$, in contrast to $\mbox{Tot}(C^*)$ which is how it is typically denoted (for instance see~\cite{weibel}). The homogeneous degree of elements in graded spaces will be denoted by straight brackets $|\cdot|$.
\newline\newline\noindent
Let $A$ be an associative $\Bbbk$-algebra. Then the \emph{Hochschild complex} of $A$ with coefficients in $M$, denoted $C^\ast(A, M)$ is the cochain complex
\[
\begin{tikzcd}
M \ar{r}{b} & \Hom{\Bbbk}{A}{M} \ar{r}{b} & \Hom{\Bbbk}{A^{\otimes 2}}{M}  \ar{r} & \cdots
\end{tikzcd}
\]
where $b$ is called the \emph{Hochschild coboundary map} and is given by $b = \sum_{i = 0}^n(-1)^i\partial^i$, where $\partial^i$ are the \emph{coface maps} $\map{\partial^i}{\Hom{\Bbbk}{A^{\otimes n-1}}{M}}{\Hom{\Bbbk}{A^{\otimes n}}{M}}$ defined by
\[ \partial^i(f)(a_1\otimes\dots\otimes a_{n})=
	\begin{cases}
		a_1f(a_2\otimes\dots\otimes a_{n}) & \mbox{ if } i = 0 \\
		f(a_1\otimes\dots\otimes a_{i}a_{i+1}\otimes\dots\otimes a_{n}) &\mbox{ if } 1\le i< n \\
		f(a_1\otimes\dots\otimes a_{n-1})a_{n} &\mbox{ if } i = n  
	\end{cases}
\]
If $M = A^*= \Hom{\Bbbk}{A}{\Bbbk}$ then we write $C^*(A):= C^*(A, A^*)$, and in this case 
\[
	C^n(A)=\Hom{\Bbbk}{A^{\otimes n +1}}{\Bbbk}
\]
\begin{defn}
The \emph{Hochschild cohomology $HH^\ast(A, M)$ of $A$ with coefficients in $M$} is defined as
\[
	HH^n(A, M) = H^n(C^\ast(A, M))
\]
\end{defn}

If $(A, \phi)$ is a differential graded algebra (dga) then each $\Hom{\Bbbk}{A^{\otimes n}}{M}$ is a dg-module, so then we have the bicomplex
\[
\begin{tikzcd}
\vdots & \vdots & \vdots \\
M^2 \arrow{u} \arrow{r}{b} & \Hom{\Bbbk}{A}{M}^2 \arrow{u} \arrow{r}{b} & \Hom{\Bbbk}{A^{\otimes 2}}{M}^2 \arrow{u} \arrow{r} & \cdots \\
M^1 \arrow{u}{\del} \arrow{r}{b} & \Hom{\Bbbk}{A}{M}^1 \arrow{u}{\del} \arrow{r}{b} & \Hom{\Bbbk}{A^{\otimes 2}}{M}^1 \arrow{u}{\del} \arrow{r} & \cdots \\
M^0 \arrow{u}{\del} \arrow{r}{b} & \Hom{\Bbbk}{A}{M}^0 \arrow{u}{\del} \arrow{r}{b} &  \Hom{\Bbbk}{A^{\otimes 2}}{M}^0 \arrow{u}{\del} \arrow{r} & \cdots
\end{tikzcd}
\] 
where $\map{\del}{\Hom{\Bbbk}{A^{\otimes p}}{A}^q}{\Hom{\Bbbk}{A^{\otimes p}}{A}^{q+1}}$ is given by
\[
	\del^{p,q}(f) = \phi\circ f - (-1)^qf\circ \phi^{\otimes p}.
\]
and $\phi^{\otimes p}$ is the differential of $A^{\otimes p}$. We will denote this bicomplex $EH(A, M)$ and refer to it as the \emph{Hochschild bicomplex}. So then
\[
	EH(A, M)^{p, q} = \Hom{\Bbbk}{A^{\otimes p}}{M}^q.
\] 

\begin{defn}
The \emph{Hochschild complex $C_{dg}(A, M)$} of the dg-algebra $A$ is the total complex $\mc{T}(EH(A, M))$ of the Hochschild bicomplex. The \emph{Hochschild cohomology $HH^\ast_{dg}(A, M)$} of the dg-algebra $A$ with coefficients in $M$ is then 
\[
	HH^n_{dg}(A, M) =  H^n(C_{dg}(A, M))
\]
When $M = A^\ast$ then we denote $HH^\ast_{dg}(A) := HH^\ast_{dg}(A, A^\ast)$.
\end{defn} 

\section{Cocyclic Modules}

\begin{defn}
Let $\Delta$ be the category consisting of finite ordered sets $[n] = \{0<1<\dots< n\}$, where $n\in\mb{N}$, and whose morphisms are non-decreasing maps. We call $\Delta$ the \emph{simplicial category}. 
\end{defn}

\begin{defn}
Let $R$ be a commutative unital ring. A \textit{cosimplicial $R$-module} is a covariant functor $\map{M}{\Delta}{\textbf{Mod}_R}$. The module $M([n])$ will be written $M^n$.
\end{defn}

Dually we also have the notion of a \emph{simplicial $R$-module}, which is a contravariant functor $\map{M}{\Delta}{\textbf{Mod}_R}$.

For two arbitrary objects $[n]$ and $[m]$ in $\Delta$ there are $\binom{m+n+1}{n+1}$ morphisms $[n]\to [m]$. To make sense of these morphisms we introduce the following maps.

\begin{defn}
For each $n$ and $i = 0, \dots, n$, the \textit{coface maps} $\map{\epsilon^i}{[n-1]}{[n]}$ are given by
\[
	\epsilon^i(j) =
	\begin{cases}
		j & \mbox{ if } j < i \\
		j+1 & \mbox{ if } j\ge i
	\end{cases}
\]
and the \textit{codegeneracy maps} $\map{\eta^i}{[n+1]}{[n]}$ are given by
\[
	\eta^i(j) =
	\begin{cases}
		j &\mbox{ if } j\le i \\
		j-1 & \mbox{ if } j>i
	\end{cases}
\]
\end{defn}

\begin{lem}
The coface and codegeneracy maps satisfy the ``\textit{cosimplicial identities}'':
\begin{align*}
	&\epsilon^j\epsilon^i = \epsilon^i\epsilon^{j-1} \quad\mbox{ if } i<j\\
	&\eta^j\eta^i = \eta^{i-1}\eta^j\quad\mbox{ if } i> j \\
	&\eta^j\epsilon^i =
		\begin{cases}
			\epsilon^i\eta^{j-1} \mbox{ if } i< j\\
			\bsym{1}_{[n-1]} \mbox{ if } i=j, \mbox{ or } i = j+1\\
			\epsilon^{i-1}\eta^j \mbox{ if } i>j+1
		\end{cases}
\end{align*}
\end{lem}

\begin{proof}
This can be seen by manually checking both sides of each identity.
\end{proof}

Consider an arbitrary morphism $\map{f}{[n]}{[m]}$. The only isomorphisms in $\Delta$ are identities, so all non-trivial morphisms either fail to be injective, or fail to be surjective. 

So what does it mean for $f$ to be non-injective? It means that there are two elements $i, j\in [n]$ such that $i\ne j$, then $f(i) = f(j)$. Without loss of generality we may assume that $i < j$. The restriction that $f$ preserves order forces all the elements $\{i +1, i +2, \dots, j-1\}$ to also be sent to the same element by $f$. In other words, $f$ fails to be injective if and only if there is at least one element $i$ such that $f(i) = f(i+1)$.  

The failure of surjectivity is as it is with any other typical map. 

So $f$ is characterised by its image in $[m]$, and the set of elements $\{j_1, \dots, j_h\}$ in $[n]$ such that $f(j_k) = f(j_k + 1)$.

\begin{prop}
Every morphism $\map{f}{[n]}{[m]}$ in $\Delta$ has a unique factorisation
\[
	f = \epsilon^{i_1}\cdots\epsilon^{i_k}\eta^{j_1}\cdots\eta^{j_s}
\] 
where $n-s+k=m$ and
\[
	0\le i_k < \dots < i_1 < m, \quad 0\le j_1< \dots < j_s < n-1
\]
\end{prop}

\begin{proof}
From the above discussion we observe that any morphism $\map{f}{[n]}{[m]}$ can be split into $f = gh$ where $h$ is surjective and characterises the non-injectivity of $f$, and $g$ is injective and characterises the non-surjectivity of $f$. That is, we have the following
\[
\begin{tikzcd}[]
\lbrack n\rbrack \arrow[r, twoheadrightarrow, "h"] & \lbrack\ell\rbrack \arrow[r, hook, "g"] & \lbrack m\rbrack 
\end{tikzcd}
\]
where $\ell$ is the number of unique elements in the image of $f$. There are $n-\ell=:s$ elements $j$ in $[n]$ such that $h(j)=h(j+1)$ and we label these elements $j_1 < \dots < j_s$. We see that $\eta^{j_1}\cdots\eta^{j_s}$ succeeds in merging the elements $j_1,\dots, j_s$ in $[\ell]$ in the appropriate way, and so
\[
	h = \eta^{j_1}\cdots\eta^{j_s}.
\]
Now, $g$ is injective so there are $m-\ell=:k$ elements missing in the image of $g$, and we label these elements $ i_k < \dots < i_1$. Again, we see that $\epsilon^{i_1}\cdots\epsilon^{i_k}$ succeeds in ensuring that the elements $i_1,\dots, i_k$ are not included in the image of $g$, that is, $i_1,\dots, i_k \notin g([\ell])$. As an example, consider the map $\map{f}{[4]}{[5]}$ given by $f(1) = f(2) = 2$ and $f(3) = f(4) = 4$. Then we can visualise the splitting of $f$ into $f=gh$ by
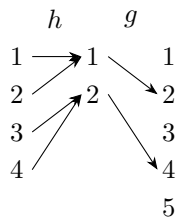
\begin{center}
 \begin{tikzpicture}[>=Stealth]

\node at (0.5, 2) {$h$};
\node at (1.5, 2) {$g$};

%[4]
\node at (0, 0) {$4$};
\node at (0, 0.5) {$3$};
\node at (0, 1) {$2$};
\node at (0, 1.5) {$1$};

%[2]
\node at (1, 1) {$2$};
\node at (1, 1.5) {$1$};

%[5]
\node at (2, -0.5) {$5$};
\node at (2, 0) {$4$};
\node at (2, 0.5) {$3$};
\node at (2, 1) {$2$};
\node at (2, 1.5) {$1$};

%h arrows
\draw[->] (0.2, 0) -- (0.85, 1);
\draw[->] (0.2, 0.5) -- (0.85, 1);
\draw[->] (0.2, 1) -- (0.85, 1.5);
\draw[->] (0.2, 1.5) -- (0.85, 1.5);

%g arrows
\draw[->] (1.2, 1.5) -- (1.85, 1);
\draw[->] (1.2, 1) -- (1.85, 0);

 \end{tikzpicture}
\end{center}
and in this case the elements $1, 3, 5\in[5]$ are the ones missing from the image of $g$. So, in this case then we can set $g=\epsilon^{5}\epsilon^{3}\epsilon^{1}$ and applying this to $[2]$ we get $\epsilon^{5}\epsilon^{3}\epsilon^{1}(1) = \epsilon^{5}\epsilon^{3}(2) = 2$ and $\epsilon^{5}\epsilon^{3}\epsilon^{1}(2) = \epsilon^{5}\epsilon^{3}(3) = \epsilon^{5}(4) = 4$.

So in general we have
\[
	g = \epsilon^{i_1}\cdots\epsilon^{i_k}
\]
and so 
\[
	f = \epsilon^{i_1}\cdots\epsilon^{i_k}\eta^{j_1}\cdots\eta^{j_s}
\]
furthermore, we have $n-\ell -(m -\ell) = h - k$ and so $n-h+k = m$. 

This representation is unique because any other representation can be put into the form above by applying the cosimplicial identities. 
\end{proof}

\begin{prop}
A module $M$ in $\textbf{Mod}_R$ is a cosimplicial module if and only if $M$ can be represented as a sequence of modules $(M^0$, $M^1, \dots)$ together with coface maps $\map{\partial^i}{M^{n-1}}{M^n}$ and codegeneracy maps $\map{\sigma^i}{M^{n+1}}{M^n}$ for $i = 0, \dots, n$, that satisfy the cosimplicial identities. 
\end{prop}

This proposition is sometimes given as the definition of a cosimplicial object \cite[page 267]{BousfieldA.K.1972Hlca}, in our case we are restricting our focus to cosimplicial objects in the category of $R$-modules.

\begin{proof}
Suppose $M$ is a cosimplicial module. Then set $M^n = M([n])$, $\partial^i = M(\epsilon^i)$ and $\sigma^i = M(\eta^i)$. Since $M$ is a covariant functor and $\epsilon^i$ and $\eta^i$ satisfy the cosimplicial identities, then so do $\partial^i$ and $\sigma^i$.

Conversely, suppose we have a sequence of modules $(M^0, M^1, \dots)$ and coface and codegeneracy maps $\partial^i$ and $\sigma^i$ respectively, satisfying the cosimplicial identities. Then we set $M([n]) = M^n$, and for some morphism $f =  \epsilon^{i_1}\cdots\eta^{j_s}$ we set $M(f) = \partial^{i_1}\cdots\sigma^{j_s}$. Since $\epsilon^i, \eta^i, \partial^i$ and $\sigma^i$ all satisfy the cosimplicial identities then the rules governing their composition is the same, and therefore $M$ is a covariant functor. 
\end{proof}

\begin{defn}
Let $M$ and $N$ be cosimplicial modules. The \emph{product cosimplicial module $M\times N$} is the cosimplicial module with $(M\times N)^n = M^n\times N^n$ and coface maps 
\[
	\partial^i(m, n) = (\partial^i(m), \partial^i(n))
\]
and codegenerancy maps
\[
	\sigma^i(m, n) = (\sigma^i(m), \sigma^i(n))
\]
for $m\in M$ and $n\in N$.
\end{defn}

According to Dold \cite{DoldAlbrecht1958HoSP} there is an equivalence of categories between the category of simplicial modules and the category of chain complexes of modules. This correspondence was independently discovered by Kan \cite{DanielM.Kan1958Ficc}, so is sometimes called the \textit{Dold-Kan Correspondence}. 

In other words, to each simplicial module, there is a uniquely associated chain complex. Dually this means that for every cosimplicial module there is a cochain complex. We can form a cochain complex from a cosimplicial module $M$ in the following way. 

\begin{prop}
\label{cosimp is cochain}
Let $M$ be a cosimplicial module with $\map{\partial^i}{M^{n-1}}{M^n}$ its coface maps and let $b^n$ be given by
\[
	b^n = \sum_{i = 0}^n(-1)^i\partial^i
\]
Then $(M, b)$, with $b = \{b^n\}$, is a cochain complex. 
\end{prop}

\begin{proof}
Consider $b^nb^{n-1}$, this is given by
\[
	b^nb^{n-1} = \sum_{i=0}^n\sum_{j=0}^{n-1}(-1)^{i+j}\partial^i\partial^j
\]
we can split this sum in the following way
\[
	b^nb^{n-1} = \sum_{i\le j}(-1)^{i+j}\partial^i\partial^j + \sum_{i>j}(-1)^{i+j}\partial^i\partial^j
\]
from the identity $\partial^i\partial^j = \partial^j\partial^{i-1}$ for $j<i$ we see that the term $(-1)^{i+j}\partial^i\partial^j$ from the first sum cancels with $(-1)^{i+1+j}\partial^{i+1}\partial^j$ in the second sum. Therefore $b^nb^{n-1} = 0$.
\end{proof}

Therefore, we can calculate the cohomology $H^\ast(M)$ of the cochain complex $(M, b)$ associated to the cosimplicial $R$-module $M$, and we refer to it as the \emph{cohomology of the cosimplicial module M}.

\begin{defn}
\label{cocyclic mod}
A \emph{cocyclic module} $M$ is a cosimplicial module with an action of the cyclic group $\mb{Z}/(n+1)\mb{Z}$ for each $M^n$, which we denote $t^n$, such that the following identities hold
\begin{align*}
	&(t^n)^{n+1} = \bsym{1}_{M^n} \\
	&t^n\partial^i = -\partial^{i-1}t^{n-1} \mbox{ and } t^n\sigma^i = -\sigma^{i-1}t^{n+1} \mbox{ for } 1\le i\le n\\
	&t^n\partial^0 = (-1)^n\partial^n \mbox{ and } t^n\sigma^0 = (-1)^n\sigma^n(t^{n+1})^2 
\end{align*}
where $\map{\partial^i}{M^{n-1}}{M^n}$ and $\map{\sigma^i}{M^{n+1}}{M^n}$ are the $i$-th coface and codegeneracy maps respectively. 
\end{defn}

We will denote $t^n=: t$ for all $n$, to simplify notation. We may also define a cocyclic module as a covariant functor from the cyclic category \cite{connes1983cohomologie}, \cite{ConnesAlain1994Ng}.

Before we define the cyclic cohomology of a cocyclic module, we need two more maps. First, we define the norm operator $\map{N}{M^n}{M^n}$ as $N = \sum_{i=0}^nt^i$ and $\map{b^\prime}{M^{n-1}}{M^n}$ as $b^\pr = \sum_{i=0}^{n-1}(-1)^i\partial^i$. 

\begin{lem}
\label{cocyc idents}
Let $b$ be the differential defined in Proposition~\ref{cosimp is cochain}. Then we have the following identities:
\begin{enumerate}[i)]
\item $(\bsym{1}-t)b = b^\pr(\bsym{1}-t)$. 
\item $Nb^\pr = bN$.
\item $N(\bsym{1}-t) = (\bsym{1}-t)N = 0$.
\end{enumerate}
where $\bsym{1}$ is the identity map. 
\end{lem}

\begin{proof}
i) From the definition of cocylic modules we know that $t\partial^i = -\partial^{i-1}t$ and $t\partial^0 = (-1)^n\partial^n$. Multiplying both sides of the first of these two by $(-1)^i$, we obtain $t((-1)^i\partial^i) = ((-1)^{i-1}\partial^{i-1})t$ and then summing from $i=1$ to $i =n$ we obtain $t(b-\partial^0) = b^\pr t$. Then from the second observation $tb = b^\pr t + (-1)^n\partial^n$ and since $b - b^\pr = (-1)^n\partial^n$ then we have
\[
	b^\pr (\bsym{1}-t) = (\bsym{1}-t)b
\]
as required.

ii) Now, suppose $i\ge j$, then recursively applying $t\partial^i = -\partial^{i-1}t$ to $t^j\partial^i$ gives
\[
	t^j\partial^i = (-1)^j\partial^{i-j}t^j
\]
and when $i<j$ then
\[
	t^j\partial^i = (-1)^it^{j-i-1}t\partial^0t^i
\]
and so applying $t\partial^0 = (-1)^n\partial^n$ and then simplifying we obtain
\[
	t^j\partial^i = (-1)^{n+j-1}\partial^{n+1+i-j}t^{j-1}
\]
So we can write
\begin{align*}
	 Nb^\pr= \left(\sum_{j=0}^{n}t^j\right)\left(\sum_{i=0}^{n-1}(-1)^i\partial^i\right)  =& \sum_{i\ge j } (-1)^{i-j}\partial^{i-j}t^{j} \\
									    	  &\; + \sum_{i < j}(-1)^{n+1+i-j}\partial^{n+1+i-j}t^{j-1}
\end{align*}  
For some $t^k$ where $0\le k < n$, then its coefficient is given by 
\[
		\sum_{0\le i <n-k}(-1)^i\partial^i + \sum_{n-k\le i\le n}(-1)^i\partial^i
\]
which is the formula for $b$. Therefore $Nb^\pr = bN$.

iii) First recall that $N = \sum_{i=0}^nt^i$, and this can be written 
\[
	\sum_{i=0}^nt^i = t^0 + \sum_{i=1}^nt^i = \bsym{1}_{M^n} +  \sum_{i=1}^nt^i
\]
and observe that 
\[
	Nt = tN =  \sum_{i=1}^{n+1}t^i =  \sum_{i=1}^nt^i + t^{n+1} =  \sum_{i=1}^nt^i +  \bsym{1}_{M^n}.
\]
Therefore $N = Nt$ and $N = tN$, and so $N(\bsym{1}-t) = (\bsym{1}-t)N = 0$.
\end{proof}
As a direct consequence of this lemma, the following diagram
\[
\begin{tikzcd}
\tvdots &\tvdots & \tvdots & \tvdots& \\   
\arrow{u} M^2 \arrow{r}{\bsym{1}-t}&\arrow{u} M^2 \arrow{r}{N}&\arrow{u} M^2 \arrow{r}{\bsym{1}-t}&\arrow{u} M^2 \arrow{r}{N}& \cdots \\
\arrow{u}{b}M^1 \arrow{r}{\bsym{1}-t}&\arrow{u}{-b^{\pr}} M^1 \arrow{r}{N}&\arrow{u}{b}M^1 \arrow{r}{\bsym{1}-t}&\arrow{u}{-b^{\pr}} M^1\arrow{r}{N}& \cdots  \\
\arrow{u}{b}M^0  \arrow{r}{\bsym{1}-t}&\arrow{u}{-b^{\pr}}M^0 \arrow{r}{N}& \arrow{u}{b}M^0 \arrow{r}{\bsym{1}-t} &\arrow{u}{-b^{\pr}}M^0 \arrow{r}{N}& \cdots
\end{tikzcd}
\]
is a bicomplex. We denote this bicomplex by $CC(M)$.

\begin{defn}
The \emph{cyclic cohomology} $HC^\ast(M)$ of the cocyclic module $M$ is defined as 
\[
	HC^n(M) = H^n(\mc{T}(CC(M)))
\]
\end{defn}

\begin{prop}
\label{cochain complex acyclic}
The map $\map{s:=(-1)^{n}\sigma^n}{M^{n+1}}{M^n}$ is a contracting homotopy of the cochain complex $(M, b^\pr)$.
\end{prop}

\begin{proof}
$s$ is a contracting homotopy if $b^\pr s + sb^\pr = \bsym{1}_{M^n}$. First,
\[
	b^\pr s = (-1)^{n-1}\sum_{i=0}^{n-1}(-1)^i\partial^i\sigma^{n-1}
\]
and then after applying $\sigma^n\partial^i = \partial^i\sigma^{n-1}$ for $0\le i < n$ to each summand of 
\[
	sb^\pr = (-1)^n\sum_{i=0}^n(-1)^i\sigma^n\partial^i
\]
and noting that when $i=n$ then the summand is $(-1)^n\sigma^n\partial^i = (-1)^n\bsym{1}_{M^n}$ we obtain
\[
	sb^\pr = \bsym{1}_{M^n} + (-1)^{n}\sum_{i=0}^{n-1}(-1)^i\partial^i\sigma^{n-1}
\]
and so it is clear that $b^\pr s + sb^\pr = \bsym{1}_{M^n}$.
\end{proof}

So, the odd columns of $CC(M)$ are acyclic, and therefore do not contribute to $HC^\ast(M)$. We aim, then, to remove these columns and in the process simplify the bicomplex. 

\begin{thm}
\label{B complex}
Denote the following diagram
\[
\begin{tikzcd}
\tvdots &\tvdots& \tvdots   \\   
\arrow{u} M^2 \arrow{r}{B}&\arrow{u} M^1 \arrow{r}{B}&\arrow{u} M^0 \\
\arrow{u}{b} M^1 \arrow{r}{B}& \arrow{u}{b} M^0   \\
\arrow{u}{b} M^0
\end{tikzcd}
\]
by $\mathcal{B}(M)$, where $B = N s(\bsym{1}-t)$. Then $\mathcal{B}(M)$ is a bicomplex and $\mc{T}(\mathcal{B}(M))$ and $\mc{T}(CC(M))$ are quasi-isomorphic.
\end{thm}

\begin{proof}
We first demonstrate that $\mathcal{B}(M)$ is a bicomplex. By Lemma~\ref{cocyc idents} and Proposition~\ref{cochain complex acyclic} we have 
\begin{align*}
	Bb &= Nsb^\pr(\bsym{1}-t)\\ &= (N-Nb^\pr s)(\bsym{1}-t)\\ &=-Nb^\pr s(\bsym{1}-t) \\ &=-bB
\end{align*} 
and 
\[
	B^2 = Ns(\bsym{1}-t)Ns(\bsym{1}-t) = 0.
\] 
Therefore $\mathcal{B}(M)$ is a bicomplex.

Now, consider $\mc{T}(CC(M))$. We can write the $n$-th term of $\mc{T}(CC(M))$ as 
\begin{align*}
	\mc{T}(CC(M))^n &= (M^{n}\oplus M^{n -2}\oplus\dots\oplus M^{ n\bmod{2}}) \\
				    &\quad \oplus (M^{ n -1}\oplus M^{ n -3}\oplus\dots\oplus M^{(n-1)\bmod{2}})
\end{align*} 
where the first summand is the contribution of the even columns, and the second is the contribution of the odd columns. Comparing this with the $n$-th term of $\mc{T}(\mathcal{B}(M))$ we see that $\mc{T}(\mathcal{B}(M))^n$ corresponds with the contribution of the even columns. Writing 
\[
	 \mc{T}(CC^{\mc{O}})^n = M^{ n -1}\oplus M^{ n -3}\oplus\dots\oplus M^{(n-1)\bmod{2}}
\]
for the contribution of the odd columns (and omitting the $(M)$ as in the other bicomplexes to save on space), then we have 
\[
	\mc{T}(CC(M))^n = \mc{T}(\mathcal{B}(M))^n \oplus \mc{T}(CC^{\mc{O}})^n.
\]
So we can write $\mc{T}(CC(M))$ as
\[
\begin{tikzcd}[cramped, ampersand replacement=\&, SL/.style = {every label/.style={yshift=#1}}]
\cdots\ar{r} \&[-1.4em] \mc{T}(\mathcal{B}(M))^{n-1} \oplus \mc{T}(CC^{\mc{O}})^{n-1} \ar[SL = 15pt]{r}{\left(\begin{smallmatrix} b & N \\ \bsym{1}-t & -b^\pr \end{smallmatrix}\right)}\&[0.6em] \mc{T}(\mathcal{B}(M))^n \oplus \mc{T}(CC^{\mc{O}})^n \ar{r} \&[-1.4em]  \cdots
\end{tikzcd}
\]
Consider the map 
\[
	\map{I}{(\mc{T}(\mc{B}(M)), b+B)}{\left(\mc{T}(\mc{B}(M))\oplus \mc{T}(CC^{\mc{O}}), \begin{psmallmatrix}  b & N \\ \bsym{1}-t & -b^\pr \end{psmallmatrix}\right)}
\]
given by
\[
	I = \begin{pmatrix} \bsym{1}_{\mc{T}(\mc{B}(M))} \\ s(\bsym{1} - t) \end{pmatrix}.
\]
Note that if $I(m)=0$ then we must have $\bsym{1}_{\mc{T}(\mc{B}(M))}(m)=m=0$, and therefore $I$ is injective. We want to show that $I$ is a cochain map, so with that in mind consider
\[
	\begin{pmatrix} b & N \\ \bsym{1} -t & -b^\pr \end{pmatrix}\circ \begin{pmatrix} \bsym{1}_{\mc{T}(\mc{B}(M))} \\ s(\bsym{1} -t) \end{pmatrix} = \begin{pmatrix} b + B \\ \bsym{1}-t - b^\pr s(\bsym{1} -t)\end{pmatrix} 
\]
and 
\begin{align*}
	\begin{pmatrix} \bsym{1}_{\mc{T}(\mc{B}(M))} \\ s(\bsym{1} -t) \end{pmatrix} \circ (b + B) &= \begin{pmatrix} b + B \\ s(\bsym{1}-t)b + s(\bsym{1} -t)B\end{pmatrix} \\
															&= \begin{pmatrix} b + B \\ s(\bsym{1}-t)b + s(\bsym{1} -t)Ns(\bsym{1}-t)\end{pmatrix}  \\
															&= \begin{pmatrix} b + B \\ s(\bsym{1}-t)b\end{pmatrix}.
\end{align*}
Then $I$ commutes with the differentials if and only if 
\[
	 \bsym{1}-t  =  b^\pr s(\bsym{1} -t) + s(\bsym{1}-t)b.
\]
But from Lemma~\ref{cocyc idents} and Proposition~\ref{cochain complex acyclic} we have 
\begin{align*}
	b^\pr s(\bsym{1} -t) + s(\bsym{1}-t)b &= b^\pr s(\bsym{1} - t) + sb^\pr(\bsym{1}-t)\\
							&= (b^\pr s + sb^\pr)(\bsym{1}-t) \\
							&=\bsym{1}-t
\end{align*}
as required. So $I$ is an injective cochain map. 

Now define 
\[
	\map{P}{\left(\mc{T}(\mc{B}(M))\oplus \mc{T}(CC^{\mc{O}}),\begin{psmallmatrix}  b & N \\ \bsym{1}-t & -b^\pr \end{psmallmatrix}\right)}{(\mc{T}(CC^{\mc{O}}), -b^\pr)}
\] 
as
\[
	P(m_1, m_2) = m_2 - s(\bsym{1} - t)(m_1).
\]
Then $P$ is surjective (for $m\in \mc{T}(CC^{\mc{O}})$ choose $(0, m)\in \mc{T}(\mc{B}(M))\oplus\mc{T}(CC^{\mc{O}})$), and $\im{I} = \ker{(P)}$.  We need to show that $P\begin{psmallmatrix}  b & N \\ \bsym{1}-t & -b^\pr \end{psmallmatrix} = -b^\pr P$, so consider
\begin{align*}
	P\circ \begin{pmatrix}  b & N \\ \bsym{1}-t & -b^\pr \end{pmatrix} & = (\bsym{1}-t) -b^\pr - s(\bsym{1}-t)b -s(\bsym{1}-t)N \\
													&= (\bsym{1}-t) -b^\pr - s(\bsym{1}-t)b
\end{align*}
and
\[
	-b^\pr P = -b^\pr + b^\pr s(\bsym{1}-t).
\]
Then in order for the differentials to commute with $P$ we just need to show
\[
	(\bsym{1}-t) - s(\bsym{1}-t)b = b^\pr s(\bsym{1}-t).
\]
But we have
\begin{align*}
	(\bsym{1}-t) - s(\bsym{1}-t)b & = (\bsym{1}-t) - sb^\pr(\bsym{1}-t) \\
						& = b^\pr s(\bsym{1}-t)
\end{align*}
from Lemma~\ref{cocyc idents} and Proposition~\ref{cochain complex acyclic}. So $P$ is a surjective cochain map. Therefore we have a short exact sequence of cochain complexes 
\[
\begin{tikzcd}
    0 \arrow[r]&  \mc{T}(\mathcal{B}(M)) \arrow{r}{I} & \mc{T}(CC(M)) \arrow{r}{P} &  \mc{T}(CC^{\mc{O}}) \arrow[r] & 0
\end{tikzcd}
\]
But since $(M, b^\pr)$ is acyclic, then so is $(\mc{T}(CC^{\mc{O}}), -b^\pr)$. And so $\mc{T}(\mathcal{B}(M))$ and $\mc{T}(CC(M))$ are quasi-isomorphic.
\end{proof}

\section{Cyclic Cohomology of DGAs}

Now we consider cyclic cohomology of dgas. First note that for a non-graded algebra $A$ the Hochschild coboundary map on $C^n(A)$ is given by
\begin{align*}
	b(f)(a_0\otimes\dots\otimes a_{n}) &= \sum_{i=0}^{n-1}(-1)^if(a_0\otimes\dots\otimes a_ia_{i+1}\otimes\dots\otimes a_{n})\\
						         & \quad +(-1)^{n} f(a_{n}a_0\otimes\dots\otimes a_{n-1})
\end{align*}
and in particular this means that the coface map $\partial^n$ is given by
\[
	\partial^n(f)(a_0\otimes\dots\otimes a_n) = f(a_na_0\otimes\dots\otimes a_{n-1}).
\]
However when $(A, \phi)$ is a dga, then we have the Koszul sign rule to account for gradings, and so in this case we have
\[
	\partial^n(f)(a_0\otimes\dots\otimes a_n) = (-1)^{|a_n|(|a_0| + \dots + |a_{n-1}|)}f(a_na_0\otimes\dots\otimes a_{n-1}).
\]
Similarly we have to modify our cyclic action $t$. We denote the modified cyclic action by $\tau$, and using the Koszul sign rule the cyclic action is given by
\[
	\tau (f)(a_0\otimes\dots\otimes a_n) = (-1)^{|a_n|(|a_0| + \dots + |a_{n-1}|) + n}f(a_n\otimes a_0\otimes\dots\otimes a_{n-1}).
\]
For ease we will set $\gamma_k := |a_k|\sum_{i\ne k}|a_i|$, so then we can rewrite the modified cyclic action and coface map as
\begin{align*}
	&\partial^n(f)(a_0\otimes\dots\otimes a_n) = (-1)^{\gamma_n}f(a_na_0\otimes\dots\otimes a_{n-1})\\
	&\tau (f)(a_0\otimes\dots\otimes a_n) = (-1)^{\gamma_n + n}f(a_n\otimes a_0\otimes\dots\otimes a_{n-1}).
\end{align*}

\begin{prop}
\label{tau is cyclic}
$[n]\mapsto C^n(A)$ with the modified cyclic action $\tau$ is a cocyclic module.
\end{prop}

\begin{proof}
We aim to demonstrate that $\tau$ satisfies the cocyclic identities from Definition~\ref{cocyclic mod}. We will explicitly check 
\begin{align*}
	&\tau^{n+1} = \bsym{1}_{C^n(A)} \\
	&\tau\sig^0 = (-1)^n\sig^n\tau^2
\end{align*}
and the other cocyclic identities can be easily verified from simply checking both sides. 

Consider $\tau^{n+1}$, then we have 
\[
	\tau^{n+1}f(a_0\otimes\dots\otimes a_n) = (-1)^{\sum_{i=0}^n\gamma_i}f(a_0\otimes\dots\otimes a_n)
\]
and
\[
	\sum_{i=0}^n\gamma_i = \sum_{i=0}^n\bigg(|a_i|\sum_{j\ne i}|a_j|\bigg).
\]
This sum consists of the products of the degrees of two distinct elements in various combinations. Fix two such distinct elements, say $a_r$ and $a_t$. The product $|a_r||a_t|$ only appears in the summands $i=r$ and $i=t$, i.e. the summands
\begin{align*}
	&|a_r|\sum_{j\ne r}|a_j| \\
	&|a_t|\sum_{j\ne t}|a_j|.
\end{align*}
In other words $|a_r||a_t|$ appears only twice in $\sum_{i=0}^n\gamma_i$. So we have
\[
	\sum_{i=0}^n\gamma_i = 2\sum_{i\ne j}|a_i||a_j|
\]
and therefore $\tau^{n+1} =  \bsym{1}_{C^n(A)}$.

Now consider $\tau\sig^0$, then 
\begin{align*}
	\tau\sig^0(f)(a_0\otimes\dots\otimes a_n) &=(-1)^{\gamma_n + n} \sig^0(f)(a_n\otimes a_0\dots\otimes a_{n-1})\\
								&= (-1)^{\gamma_n + n}f(a_n\otimes 1\otimes a_0\otimes \dots\otimes a_{n-1})
\end{align*} 
and 
\[
	(-1)^n\sig^n\tau^2(f)(a_0\otimes\dots\otimes a_n)  = (-1)^n\tau^2(f)(a_0\otimes\dots\otimes a_n\otimes 1).
\]
When applying $\tau$ for the first time we have the sign $(-1)^{\gamma_{n+1} + n+1}$. In this case $\gamma_{n+1} = |1|\sum_{i\ne k}|a_i| = 0$ and therefore
\begin{align*}
	(-1)^n\sig^n\tau^2(f)(a_0\otimes\dots\otimes a_n) & = -\tau(f)(1\otimes a_0\otimes\dots\otimes a_n) \\
										& = (-1)^{\gamma_{n+1} + n}f(a_n\otimes 1\otimes a_0\otimes\dots\otimes a_{n-1}).
\end{align*}
In this case $\gamma_{n+1} = |a_n|\sum_{i\ne k} |a_i|$ and 
\[
	\sum_{i\ne k} |a_i| = |1| + |a_0| + \dots + |a_{n-1}| = |a_0| + \dots + |a_{n-1}|
\]
so $\gamma_{n+1} = \gamma_n$ and therefore
\[
	(-1)^n\sig^n\tau^2(f)(a_0\otimes\dots\otimes a_n) = (-1)^{\gamma_{n} + n}f(a_n\otimes 1\otimes a_0\otimes\dots\otimes a_{n-1})
\]
Hence $\tau\sig^0 = (-1)^n\sig^n\tau^2$.
\end{proof}

So we have a bicomplex $\mc{B}(A)$. However, each $C^n(A)$ is a dg-module with grading
\[
	C^n(A)^m = \Hom{\Bbbk}{(A^{\otimes n +1})^m}{\Bbbk}
\]
and differential $\del(f) = (-1)^mf\phi^{\otimes n+1}$. The modified cyclic action $\tau$ naturally extends to $C^n(A)^m$. 

\begin{lem}
\label{delB is Bdel}
$\del B = B\del$.
\end{lem} 

\begin{proof}
Since $B = Ns(\bsym{1}-\tau)$, and $N = \sum_{i=0}^n\tau^i$ and $s = (-1)^n\sig^n$, then it's enough to show that $\del\tau = \tau\del$ and $\del\sig^n = \sig^n\del$. 

We will first show that $\del\tau = \tau\del$. We have 
\[
	\del(f)(a_0\otimes\dots\otimes a_n) = (-1)^m\sum_{i=0}^n(-1)^{\omega_{i-1}}f(a_0\otimes\dots\otimes \phi(a_i)\otimes\dots\otimes a_n)
\]
where $\omega_{-1} = 0$ and $\omega_k = \sum_{i=0}^k|a_i|$. We will also set 
\begin{align*}
	&\gamma_k^\pr := |a_k|\bigg(1+\sum_{i\ne k}|a_i|\bigg) = |a_k| + \gamma_k\\
	&\gamma_k^{\pr\pr} := (|a_k| + 1)\sum_{i\ne k}|a_i| = \gamma_k + \sum_{i\ne k}|a_i|.
\end{align*} 
In particular we have $\gamma_n^{\pr\pr} = \gamma_n + \omega_{n-1}$. Now consider $\del\tau$, then we have
\begin{align*}
	\del\tau(f)(a_0\otimes&\dots\otimes a_n)  = (-1)^m\sum_{i=0}^n(-1)^{\omega_{i-1}}\tau(f)(a_0\otimes\dots\otimes \phi(a_i)\otimes\dots\otimes a_n) \\
							    & = (-1)^{m+n}\bigg(\sum_{i=0}^{n-1}(-1)^{\omega_{i-1} + \gamma_n^\pr}f(a_n\otimes a_0\otimes\dots\otimes \phi(a_i)\otimes\dots\otimes a_{n-1}) \\
							    &\quad + (-1)^{\omega_{n-1} + \gamma_n^{\pr\pr}}f(\phi(a_n)\otimes a_0\otimes\dots\otimes a_{n-1})\bigg).
\end{align*}
Note that 
\[
	\omega_{n-1}+ \gamma_n^{\pr\pr} = \omega_{n-1} + \gamma_n + \omega_{n-1}  =  2\omega_{n-1} + \gamma_n
\]
and therefore $(-1)^{\omega_{n-1} + \gamma_n^{\pr\pr}} = (-1)^{\gamma_n}$. So
\begin{align*}
	\del\tau(f)(a_0\otimes&\dots\otimes a_n)  \\
					 & = (-1)^{m+n+ \gamma_n}\bigg(\sum_{i=0}^{n-1}(-1)^{\omega_{i-1}+|a_n|}f(a_n\otimes a_0\otimes\dots\otimes \phi(a_i)\otimes\dots\otimes a_{n-1}) \\
							    &\quad + f(\phi(a_n)\otimes a_0\otimes\dots\otimes a_{n-1})\bigg).
\end{align*}
Now consider $\tau\del$, then we have
\begin{align*}
	\tau\del(f)(a_0\otimes&\dots\otimes a_n) = (-1)^{\gamma_n + n}\del(f)(a_n\otimes a_0\otimes\dot\otimes a_{n-1})\\
					& = (-1)^{m+n+\gamma_n}\bigg(\sum_{i=0}^{n-1}(-1)^{\omega^\pr_i}f(a_n\otimes a_0\otimes\dots\otimes \phi(a_i)\otimes\dots\otimes a_{n-1})\\
					&\quad + f(\phi(a_n)\otimes a_0\otimes\dots\otimes a_{n-1})\bigg) 
\end{align*}
where $\omega^\pr_i = |a_n| + |a_0|+\dots+|a_{i-1}| = \omega_{i-1} + |a_n|$. Therefore $\del\tau = \tau\del$. 

Now consider $\sig^n\del$, then we have 
\begin{align*}
	\sig^n\del(f)(a_0\otimes\dots\otimes a_n) &= \del(f)(a_0\otimes\dots\otimes a_n\otimes 1) \\
								& = (-1)^m\sum_{i=0}^n(-1)^{\omega_{i-1}}f(a_0\otimes\dots\otimes \phi(a_i)\otimes\dots\otimes a_n\otimes 1) \\
								&\quad + (-1)^{m+\omega_{n}}f(a_0\otimes\dots\otimes a_n\otimes \phi(1)) \\
								& = (-1)^m\sum_{i=0}^n(-1)^{\omega_{i-1}}f(a_0\otimes\dots\otimes \phi(a_i)\otimes\dots\otimes a_n\otimes 1) \\
								& = \del\sig^n(f)(a_0\otimes\dots\otimes a_n).
\end{align*}
We can kill the term involving $\phi(1)$ because the unit in $A$ is a $0$-cocycle, since $\ker{\phi}$ is a graded-subalgebra of $A$ and therefore must contain the multiplicative identity in the $0$-th degree. Therefore $\del\sig^n = \sig^n\del$ and so $\del B = B\del$ as required.
\end{proof}

Therefore $\mc{B}(A)$ expands as a tricomplex, which we denote by $EC(A)$, with differentials $b, B$ and $\pm\del$, i.e, the sign of $\del$ alternates.

\begin{defn}
\label{tri cyc cohom of dga}
The \emph{cyclic cohomology $HC^*_{dg}(A)$} of the dg-algebra $A$ is the cohomology of the total complex of the tricomplex $EC(A)$, so that 
\[
	HC^n_{dg}(A) = H^n(\mc{T}(EC(A)))
\]
for each $n$.
\end{defn}

\section{Spectral Sequences for the Cyclic Cohomology of a DGA}

There is currently little literature that explores spectral sequences for cyclic cohomology of dgas. The focus in the literature appears to be on spectral sequences for periodic cyclic homology (see~\cite{Kaledin2017} for example), which we do not discuss here.   

When we set coefficients as $A^\ast = \Hom{\Bbbk}{A}{\Bbbk}$ then the Hochschild complex $C^*(A) := C^*(A, A^\ast)$ is given by
\[
	C^n(A)=\Hom{\Bbbk}{A^{\otimes n +1}}{\Bbbk}
\]
For a dga $(A, \phi)$ then $C^*(A)$ expands as the bicomplex $EH(A) := EH(A, A^\ast)$ given by
\[
	EH(A)^{n,m} = \Hom{\Bbbk}{(A^{\otimes n + 1})^m}{\Bbbk}
\]
and we have a differential $\map{\del}{EH(A)^{n,m}}{EH(A)^{n,m+1}}$ given by $\del(f) = (-1)^mf\phi^{\otimes n+1}$ where $\phi^{\otimes n+1}$ is the differential on $A^{\otimes n+1}$. 

Now define the map $\map{\lambda}{C^n(A)}{C^n(A)}$ as 
\[
	\lam f(a_0\otimes\dots\otimes a_n) = (-1)^nf(a_n\otimes a_0 \otimes\dots\otimes a_{n-1})
\]
then we set 
\[
	C^n_\lam(A) := \ker{(\bsym{1}-\lam)}
\]
and $C_{\lam}(A)$ is a subcomplex of the Hochschild cochain complex. Now, $\lam$ induces a map $\map{\Lam}{EH(A)^{n,m}}{EH(A)^{n,m}}$ defined in the same way, except due to the Koszul sign rule we need to introduce an additional sign given by
\[
	\Lam f(a_0\otimes\dots\otimes a_n) = (-1)^{|a_n|(|a_0| +\dots + |a_{n-1}|) + n}f(a_n\otimes a_0 \otimes\dots\otimes a_{n-1})
\]
and we define $C^n_{\Lam}(A)^m$ as
\[
	C^n_{\Lam}(A)^m = \ker{(\bsym{1}-\Lam)}
\]
Clearly $C^n_{\Lam}(A)^m\subset EH(A)^{n,m}$. Furthermore, taking $C_{\Lam}(A)$ as a bigraded $\Bbbk$-module, we have the following lemma. 

\begin{lem}
$(\mc{T}(C_{\Lam}(A)), d)$ is a subcomplex of $(\mc{T}(EH(A)), d)$.
\end{lem}

\begin{proof}
Since $C^n_{\Lam}(A)^m\subset EH(A)^{n,m}$ then we immediately have $\mc{T}(C_{\Lam}(A))^n\subset\teha{A}^n$. And since $\Lam$ is equivalent to $\tau$ then by Proposition~\ref{tau is cyclic} $\Lam$ is also a cyclic action satisfying the cocyclic identities. Therefore by Lemma~\ref{cocyc idents} we have $(\bsym{1}-\Lam)b = b^\pr(\bsym{1}-\Lam)$. Then if $f\in C^n_\lam(A)$ then 
\[
	(\bsym{1}-\Lam)b(f) = b^\pr (\bsym{1}-\Lam)(f) = 0
\]
and therefore $b(f)\in\ker{(\bsym{1}-\Lam)} = C^{n+1}_\Lam(A)$. Therefore $b(C^n_{\Lam}(A)^m)\subset C^{n+1}_{\Lam}(A)^m$. Taking $f\in C^n_{\Lam}(A)^m$ we see that
\[
	(\bsym{1}-\Lam)\del(f) = (-1)^m(\bsym{1}-\Lam)f\phi^{\otimes n+1} = 0	
\]
and so $\del(f)\in C^n_{\Lam}(A)^{m+1}$. Therefore $\del(C^n_{\Lam}(A)^m)\subset C^n_{\Lam}(A)^{m+1}$. Hence $d = \sum_{p+q=n}b +(-1)^p\del$ is also a well defined differential on $\mc{T}(C_{\Lam}(A))$ and thus $(\mc{T}(C_{\Lam}(A)), d)$ is a subcomplex of $(\mc{T}(EH(A)), d)$.
\end{proof}

Therefore $C_{\Lam}(A)$ is a bicomplex
\[
\begin{tikzcd}
\vdots & \vdots & \vdots \\
C^0_{\Lam}(A)^2\arrow{u} \arrow{r}{b} & C^1_{\Lam}(A)^2 \arrow{u} \arrow{r}{b} & C^2_{\Lam}(A)^2 \arrow{u} \arrow{r} & \cdots \\
C^0_{\Lam}(A)^1 \arrow{u}{\del} \arrow{r}{b} & C^1_{\Lam}(A)^1 \arrow{u}{-\del} \arrow{r}{b} & C^2_{\Lam}(A)^1 \arrow{u}{\del} \arrow{r} & \cdots \\
C^0_{\Lam}(A)^0 \arrow{u}{\del} \arrow{r}{b} & C^1_{\Lam}(A)^0 \arrow{u}{-\del} \arrow{r}{b} &  C^2_{\Lam}(A)^0 \arrow{u}{\del} \arrow{r} & \cdots
\end{tikzcd}
\] 
and so we can define the cyclic cohomology as follows. 

\begin{defn}
\label{lam hc of dga}
For a dga $(A, \phi)$, the \emph{cyclic cohomology} $HC^*_{dg}(A)$ is given by
\[
	HC^*_{dg}(A) = H^*(\mc{T}(C_{\Lam}(A))).
\]
\end{defn}

Note that this definition differs from Definition~\ref{tri cyc cohom of dga}. As we will see in Theorem~\ref{teca iso to tcalam} then when $A$ is a dga over a field $\Bbbk$ of characteristic $0$ then these two definitions are equivalent. We will briefly recall Definition~\ref{tri cyc cohom of dga}, and recall that for a dga $(A, \phi)$ the cyclic bicomplex $\mc{B}(A)$
\[
\begin{tikzcd}
\tvdots &\tvdots &\tvdots   \\   
\arrow{u} C^2(A) \arrow{r}{B}&\arrow{u} C^1(A) \arrow{r}{B}&\arrow{u} C^0(A) \\
\arrow{u}{b} C^1(A)  \arrow{r}{B}& \arrow{u}{b} C^0(A)    \\
\arrow{u}{b} C^0(A) 
\end{tikzcd}
\]
expands as the tricomplex $EC(A)$ given by
\[
\begin{tikzpicture}[commutative diagrams/every diagram, on top/.style={preaction={draw=white,-,line width=#1}},
on top/.default=6pt]
\node[rotate=-20]  (1) at (0.25,8.25) {$\ddots$}   ;
\node  (2) at (1,8.25) {$\tvdots$}  ;
\node[rotate=-20]  (23) at (3.25,8.25) {$\ddots$}   ;
\node  (3) at (4,8.25) {$\tvdots$}  ;
\node[rotate=-20]  (24) at (6.25,8.25) {$\ddots$}   ;
\node  (4) at (7,8.25) {$\tvdots$}  ;
\node  (5) at (1,7) {$C^2(A)^1$}   ;
\node  (6) at (4,7) {$C^1(A)^1$}   ;
\node  (7) at (7,7) {$C^0(A)^1$}   ;
\node[rotate=-20]  (8) at (0,6) {$\ddots$}   ;
\node  (9) at (2,6.25) {$\tvdots$}   ;
\node[rotate=-20] (10) at (3,6) {$\ddots$}   ;
\node  (11) at (5,6.25) {$\tvdots$}   ;
\node[rotate=-20] (12) at (6,6) {$\ddots$}   ;
\node  (13) at (2,5) {$C^1(A)^1$}   ;
\node  (14) at (5,5) {$C^0(A)^1$}   ;
\node  (15) at (1,4) {$C^2(A)^0$}   ;
\node  (16) at (3,4.25) {$\tvdots$}   ;
\node  (17) at (4,4) {$C^1(A)^0$}   ;
\node  (18) at (7,4) {$C^0(A)^0$}   ;
\node  (19) at (3,3) {$C^0(A)^1$}   ;
\node  (20) at (2,2) {$C^1(A)^0$}   ;
\node  (21) at (5,2) {$C^0(A)^0$}   ;
\node  (22) at (3,0) {$C^0(A)^0$}   ;

\path[commutative diagrams/.cd, every arrow, every label]
       (22) edge node {$b$} (20)
       (20) edge node {$b$} (15)
       (20) edge node[swap] {$B$} (21)
       (15) edge (8)
       (15) edge (5)
       (15) edge (17)
       (21) edge (17)
       (17) edge (18)
       (13) edge (5)
       (5) edge (6)
       (6) edge (7)
       (17) edge (6)
       (18) edge (7)
       (14) edge (6)
       (19) edge (16)
       (13) edge (9)
       (14) edge (11)
       (17) edge (10)
       (18) edge (12)
       (5) edge (1)
       (5) edge (2)
       (6) edge (3)
       (7) edge (4)
       (6) edge (23)
       (7) edge (24)
       
       (22) edge[on top] node {$\del$} (19)
       (19) edge[on top] (13)
       (20) edge[on top] node {$-\del$} (13)
       (13) edge[on top] (14)
       (21) edge[on top] (14);
\end{tikzpicture}
\]
where $C^n(A)^m =  \Hom{\Bbbk}{(A^{\otimes n + 1})^m}{\Bbbk}$ and $\map{\del}{C^n(A)^{m}}{C^n(A)^{m+1}}$ is given by $\del(f) = (-1)^mf\phi^{\otimes n+1}$. Note that this tricomplex exists for dgas over any arbitrary ring. See~\cite{GorokhovskyAlexander2002Scca} for more details on the construction of this cyclic complex.

Restating Definition~\ref{tri cyc cohom of dga} we have:

\begin{defn}
\label{hc defn with cyc tricomp}
The cyclic cohomology $HC^*_{dg}(A)$ of $(A, \phi)$ is given by
\[
	HC^*_{dg}(A) = H^*(\mc{T}(EC(A))).
\]
\end{defn}
Note how, when indexing $EC(A)$ by $EC(A)^{p,q,r}$, then each $EC(A)^{p, \ast,\ast}$ is the Hochschild bicomplex $EH(A)$ but shifted up in the $q$ direction by $p$. This leads to an interesting decomposition of each component of $\mc{T}(EC(A))$ as follows.
\begin{lem}
\label{int decomp}
For each $n$ we have
\[
\mc{T}(EC(A))^n = 
\begin{dcases}
\bigoplus_{i=0}^{\frac{n}{2}}\mc{T}(EH(A))^{2i}&$n$\mbox{ is even}\\
\bigoplus_{i=0}^{\frac{n-1}{2}}\mc{T}(EH(A))^{2i+1}&$n$\mbox{ is odd}
\end{dcases}
\]
\end{lem}

\begin{proof}
Each component $\mc{T}(EC(A))^n$ is made up of the points of $EC(A)$ that intersect with the plane $p+q+r = n$. Inspecting how this plane intersects with $EC(A)$ we see that the contribution of the Hochschild bicomplex $EC(A)^{0,\ast,\ast}$ to $\mc{T}(EC(A))^n$ is $\bigoplus_{q+r=n}EC^{0,q,r} = \mc{T}(EH(A))^n$. However, since $EC(A)^{1,\ast,\ast}$ is the Hochschild bicomplex but shifted up in the $q$ direction by $1$ then by the time the $p+q+r=n$ place intersects with $EC(A)^{1,\ast,\ast}$ it's moved down $2$ in the $r$ direction, meaning that the contribution of $EC^{1,\ast,\ast}$ to $\mc{T}(EC(A))^n$ is $\bigoplus_{q+r=n-2}EC(A)^{1,q,r} = \mc{T}(EH(A))^{n-2}$. We continue in this way until we reach the final contribution, which depending on whether $n$ is odd or even will either be $C^1(A)^0\oplus C^0(A)^1 = \mc{T}(EH(A))^1$ if odd, or $C^0(A)^0 = \mc{T}(EH(A))^0$ if even.
\end{proof}

Now, from Theorem~\ref{B complex} we have the quasi-isomorphism $\map{I}{\mc{T}(\mc{B}(A))}{\mc{T}(CC(A))}$ when forgetting the differential graded structure of $A$. Remembering the dg structure of $A$ then of course $\mc{B}(A)$ expands into $EC(A)$, and $CC(A)$ also expands into a tricomplex, which we will denote by $CC_{dg}(A)$.

Just as $\map{\lam}{C^n(A)}{C^n(A)}$ induces $\map{\Lam}{C^n(A)^m}{C^n(A)^m}$, then we also have that $\map{N}{C^n(A)}{C^n(A)}$ induces $\map{\nabla}{C^n(A)^m}{C^n(A)^m}$ given by
\begin{equation}
\label{nabla}
	\nabla = \sum_{i=0}^n\Lam^i.
\end{equation} 
So then $CC_{dg}(A)$ is given by 
\[
\begin{tikzcd}[cramped, sep=small]
\ddots &\tvdots & \ddots& \tvdots &\ddots&\tvdots & \\
& C^1(A)^2 \ar{rr} \ar{u}\ar{ul}     &                                  \tvdots                                  & C^1(A)^2 \ar{rr}\ar{ul}   \ar{u} \ar[from=dd]         &                            \tvdots                                      & C^1(A)^2 \ar{u} \ar{r} \ar[from=dd]   \ar{ul} & \mathrlap{\cdots}\tvdots &\\
\ddots&                                      & C^0(A)^2 \ar{ul} \ar[crossing over]{rr}   \ar{u}  &                              & C^0(A)^2 \ar{ul} \ar[crossing over]{rr} \ar{u}                                      &                                   & C^0(A)^2  \ar{ul}  \ar{u}   \ar{r} & \cdots    \\
&C^1(A)^1 \ar{rr} \ar{uu} \ar{ul} &                                                                                  & C^1(A)^1  \ar[from = dd] \ar{rr}            &                                    & C^1(A)^1  \ar{r}\ar[from=dd]           &       \cdots                                                     \\
\ddots &                                     & C^0(A)^1 \ar[crossing over]{rr} \ar{ul} \ar[crossing over]{uu} &                                               & C^0(A)^1 \ar[crossing over]{rr} \ar{ul} \ar[crossing over]{uu} &                                   & C^0(A)^1 \ar{ul} \ar{uu}    \ar{r} & \cdots          \\
&C^1(A)^0 \ar{uu}{-\del} \ar{rr}\ar{ul} &                                                                                                           & C^1(A)^0 \ar{rr} &                                                                                  & C^1(A)^0 \ar{r}&                                                       \cdots     \\
 &                                             &C^0(A)^0 \ar{ul}{b} \ar[crossing over, near start]{uu}{\del} \ar[swap]{rr}{\bsym{1}-\Lambda} &                                               & C^0(A)^0 \ar{ul}{-b^\pr} \ar[crossing over, near start]{uu}{\del} \ar[swap]{rr}{\nabla}  &       & C^0(A)^0 \ar{ul} \ar{uu}\ar{r} & \cdots \\
\end{tikzcd}
\]

\begin{thm}
\label{teca quasi to tccdga}
$\mc{T}(EC(A))$ is quasi-isomorphic to $\mc{T}(CC_{dg}(A))$.
\end{thm}

\begin{proof}
If we index $CC_{dg}(A)$ by $CC_{dg}(A)^{p,q,r}$ then for each even value of $p$ the bicomplex $CC_{dg}(A)^{p, *, *}$ is $EH(A)$. Similarly, for the odd values of $p$ the bicomplex $CC_{dg}(A)^{p, *, *}$ is $EH(A)$ except the horizontal differential is $-b^\pr$; in this case we will denote $CC_{dg}(A)^{p, *, *}$ as $EH^\pr(A)$.

By inspecting the components of $\mc{T}(CC_{dg}(A))$ we see that 
\[
	\mc{T}(CC_{dg}(A))^n = \mc{T}(EH(A))^{n}\oplus \mc{T}(EH^\pr(A))^{n-1} \oplus \mc{T}(EH(A))^{n-2}\oplus  \dots
\] 
The final summand of this direct sum is $C^0(A)^0$, however whether it comes from $EH(A)$ or $EH^\pr(A)$ depends on the parity of $n$. If $n$ is even then it comes from $EH(A)$, and if $n$ is odd then it's from $EH^\pr(A)$. Therefore we have
\[
\mc{T}(CC_{dg}(A))^n = 
\begin{dcases}
\left(\bigoplus_{i=0}^{\frac{n}{2}}\mc{T}(EH(A))^{2i}\right)\oplus\left(\bigoplus_{i=1}^{\frac{n}{2}}\mc{T}(EH^\pr(A))^{2i-1}\right) &$n$\mbox{ is even}\\
\left(\bigoplus_{i=0}^{\frac{n-1}{2}}\mc{T}(EH(A))^{2i+1}\right)\oplus\left(\bigoplus_{i=0}^{\frac{n-1}{2}}\mc{T}(EH^\pr(A))^{2i}\right)&$n$\mbox{ is odd}
\end{dcases}
\]
and so from Lemma~\ref{int decomp} we have $\mc{T}(EC(A))\subset \mc{T}(CC_{dg}(A))$. We will denote the remaining components of $\mc{T}(CC_{dg}(A))^n$ by $\mc{T}(CC_{dg}^{\mc{O}})^n$ since it is made up of the contributions from the bicomplexes of $CC_{dg}(A)$ at the odd values of $p$. So we have
\[
\mc{T}(CC_{dg}^{\mc{O}})^n = 
\begin{dcases}
\bigoplus_{i=1}^{\frac{n}{2}}\mc{T}(EH^\pr(A))^{2i-1}&$n$\mbox{ is even}\\
\bigoplus_{i=0}^{\frac{n-1}{2}}\mc{T}(EH^\pr(A))^{2i}&$n$\mbox{ is odd}
\end{dcases}
\]
and therefore
\[
	\mc{T}(CC_{dg}(A))^n = \mc{T}(EC(A))^n\oplus \mc{T}(CC_{dg}^{\mc{O}})^n.
\]
Comparing with Theorem~\ref{B complex} then $\mc{T}(EC(A))$ corresponds to $\mc{T}(\mc{B}(A))$ and $\mc{T}(CC_{dg}^{\mc{O}})$ corresponds to $\mc{T}(CC^{\mc{O}})$. So we can write the differential of $\mc{T}(CC_{dg}(A))$ as 
\[
	d = \begin{pmatrix} b\pm\del & \nabla \\ \bsym{1}-\Lam & -b^\pr\pm\del \end{pmatrix}.
\]
Now, $I$ induces the map 
\[
	\map{\Psi}{(\mc{T}(EC(A)), b+B\pm\del)}{\left(\mc{T}(EC(A))\oplus\mc{T}(CC_{dg}^{\mc{O}}), \begin{psmallmatrix} b\pm\del & \nabla \\ \bsym{1}-\Lam & -b^\pr\pm\del \end{psmallmatrix}\right)}
\]
given by
\[
	\Psi = \begin{pmatrix} \bsym{1}_{\mc{T}(EC(A))} \\ s(\bsym{1} - \Lam) \end{pmatrix}
\]
where $s = (-1)^n\sig^n$. In the same way that $I$ is injective then so is $\Psi$. We can perform the same calculation as in Theorem~\ref{B complex} to show that $\Psi$ is a cochain map except in this case we also need to show that
\[
	s(\bsym{1}-\Lam)\del = \del s(\bsym{1}-\Lam)
\]
which is indeed the case due to the proof of Lemma~\ref{delB is Bdel}. 

Similarly, $P$ from Theorem~\ref{B complex} induces a map 
\[
	\map{\Phi}{\left(\mc{T}(EC(A))\oplus\mc{T}(CC_{dg}^{\mc{O}}), \begin{psmallmatrix} b\pm\del & \nabla \\ \bsym{1}-\Lam & -b^\pr\pm\del \end{psmallmatrix}\right)}{(\mc{T}(CC_{dg}^{\mc{O}}),  - b^\pr\pm\del)}
\]
given by 
\[
	\Phi(f, g) = g - s(\bsym{1} - \Lam)(f)
\]
and as before $\Phi$ is a surjective cochain map such that $\im{\Psi} = \ker{(\Phi)}$ by the same reasoning used in Theorem~\ref{B complex} and Lemma~\ref{delB is Bdel}. Therefore we have a short exact sequence of cochain complexes
\[
\begin{tikzcd}
    0 \arrow[r]&  \mc{T}(EC(A)) \arrow{r}{\Psi} & \mc{T}(CC_{dg}(A)) \arrow{r}{\Phi} &  \mc{T}(CC^{\mc{O}}_{dg}) \arrow[r] & 0
\end{tikzcd}
\]
If we can show that $(\mc{T}(CC_{dg}^{\mc{O}}),  - b^\pr\pm\del)$ is acyclic then we will have proven our theorem. With that in mind, consider $EH^\pr(A)$ given by  
\[
\begin{tikzcd}
\vdots & \vdots & \vdots \\
C^0(A)^2\arrow{u} \arrow{r}{-b^\pr} & C^1(A)^2 \arrow{u} \arrow{r}{-b^\pr} & C^2(A)^2 \arrow{u} \arrow{r} & \cdots \\
C^0(A)^1 \arrow{u}{\del} \arrow{r}{-b^\pr} & C^1(A)^1 \arrow{u}{-\del} \arrow{r}{-b^\pr} & C^2(A)^1 \arrow{u}{\del} \arrow{r} & \cdots \\
C^0(A)^0 \arrow{u}{\del} \arrow{r}{-b^\pr} & C^1(A)^0 \arrow{u}{-\del} \arrow{r}{-b^\pr} &  C^2(A)^0 \arrow{u}{\del} \arrow{r} & \cdots
\end{tikzcd}
\] 
Denote the odd and even columns of $EH^\pr(A)$ by $EH^{\pr\mc{O}}$ and $EH^{\pr\mc{E}}$ respectively, then we can write $\mc{T}(EH^\pr(A))$ as  
\[
\begin{tikzcd}[cramped, ampersand replacement=\&, SL/.style = {every label/.style={yshift=#1}}]
\cdots\ar{r} \&[-1.4em] \mc{T}(EH^{\pr\mc{E}})^{n} \oplus \mc{T}(EH^{\pr\mc{O}})^{n} \ar[SL = 15pt]{r}{\left(\begin{smallmatrix} \del & -b^\pr \\ -b^\pr & -\del \end{smallmatrix}\right)}\&[0.6em] \mc{T}(EH^{\pr\mc{E}})^{n+1} \oplus\mc{T}(EH^{\pr\mc{O}})^{n+1} \ar{r} \&[-1.4em]  \cdots
\end{tikzcd}
\]
By Proposition~\ref{cochain complex acyclic} we have $sb^\pr + b^\pr s = \bsym{1}_{C^n(A)^m}$ and from the proof of Lemma~\ref{delB is Bdel} we have $\del\sig^n = \sig^n\del$ which implies $\del s = s\del$.  Therefore 
\begin{align*}
	-(-b^\pr + \del)s -s(-b^\pr - \del) & = b^\pr s -\del s +sb^\pr +s\del \\
							& = \bsym{1}_{C^n(A)^m} - \del s + \del s \\
							& = \bsym{1}_{C^n(A)^m}.
\end{align*}
Now,
\[
	\begin{pmatrix} 0 & -s \\ -s & 0 \end{pmatrix} \circ \begin{pmatrix} \del & -b^\pr \\ -b^\pr & -\del \end{pmatrix}  = \begin{pmatrix} -s(-b^\pr - \del) \\ -s(-b^\pr +\del) \end{pmatrix}
\]
and 
\[
	\begin{pmatrix} \del & -b^\pr \\ -b^\pr & -\del \end{pmatrix}\circ\begin{pmatrix} 0 & -s \\ -s & 0 \end{pmatrix}   = \begin{pmatrix} -(-b^\pr + \del)s \\ -(-b^\pr -\del)s \end{pmatrix}
\]
and therefore
\begin{align*}
\begin{pmatrix} \del & -b^\pr \\ -b^\pr & -\del \end{pmatrix}\circ\begin{pmatrix} 0 & -s \\ -s & 0 \end{pmatrix} + \begin{pmatrix} 0 & -s \\ -s & 0 \end{pmatrix} \circ & \begin{pmatrix} \del & -b^\pr \\ -b^\pr & -\del \end{pmatrix} \\ &=  \begin{pmatrix} -(-b^\pr + \del)s-s(-b^\pr - \del) \\ -(-b^\pr -\del)s -s(-b^\pr +\del) \end{pmatrix} \\
								& = \begin{pmatrix} \bsym{1}_{\mc{T}(EH^{\pr\mc{E}})^{n}} \\ \bsym{1}_{\mc{T}(EH^{\pr\mc{O}})^{n}} \end{pmatrix} = \bsym{1}_{\mc{T}(EH^\pr(A))^n}.
\end{align*}
So $\begin{pmatrix}0 & -s \\ -s & 0\end{pmatrix}$ is a contracting homotopy of $(\mc{T}(EH^\pr(A)), -b^\pr \pm\del)$ and thus it is acyclic. This implies that $(\mc{T}(CC_{dg}^{\mc{O}}),  - b^\pr\pm\del)$ is acyclic, and so we are done.
\end{proof}

Recall that we are working over a field $\Bbbk$ of characteristic $0$. Equivalently we may say that $\mb{Q}\subset \Bbbk$. Indeed, consider the ring homomorphism $\mb{Z}\to \Bbbk$ which sends $n\mapsto n\cdot 1$. Then $\Bbbk$ is of characteristic $0$ if and only if this map is injective. Therefore all non-zero elements of $\mb{Z}$ are sent to invertible elements of $\Bbbk$, and so it can be extended to an injective homomorphism $\mb{Q}\hookrightarrow \Bbbk$ which sends $m/n \mapsto (m\cdot 1)(n\cdot 1)^{-1}$. This fact is important for the next theorem.

\begin{thm}
\label{teca iso to tcalam}
$\mc{T}(EC(A))$ is quasi-isomorphic to $\mc{T}(C_{\Lam}(A))$.
\end{thm}

\begin{proof}
This proof is a generalisation of \cite[Theorem 3.8.1]{khalkCycCohom} in the sense that we are proving it for dgas over a field $\Bbbk$ of characteristic $0$, rather than non-graded $\mb{C}$-algebras.

By Theorem~\ref{teca quasi to tccdga} $\mc{T}(EC(A))$ is quasi-isomorphic to ${\mc{T}(CC_{dg}(A))}$. Then it will be enough to demonstrate that $\mc{T}(C_{\Lam}(A))$ is quasi-isomorphic to $\mc{T}(CC_{dg}(A))$.

Consider the map $\map{H}{C^n(A)}{C^n(A)}$ given by $H=\frac{1}{n+1}\sum_{i=0}^n (i+1)\Lam^i$, then recalling the definition of $\nabla$ (see~(\ref{nabla}) above) we have
\[
	\frac{1}{n+1}\nabla - (\bsym{1}-\Lam)H = \bsym{1}_{C^n(A)^m}.
\]
Clearly, this expression only makes sense if $\mb{Q}\subset \Bbbk$. But by the discussion above, we see that this is true for fields of characteristic $0$. Now, if $f\in \ker{(\nabla)}$ then we have
\[
	\frac{1}{n+1}\nabla(f) - (\bsym{1}-\Lam)H(f) = f \Rightarrow (\bsym{1}-\Lam)(-H(f)) = f
\]
and so $\ker{(\nabla)}\subseteq \im{\bsym{1}-\Lam}$. 

If $f\in\ker{(\bsym{1}-\Lam)}$ then $\Lam(f) = f$ and therefore $\nabla(f) = (n+1)f$. And so choosing $g = \frac{1}{n+1}f$ then $f= \nabla(g)$. So we have $\ker{(\bsym{1}-\Lam)}\subseteq \im{\nabla}$. Therefore every cochain complex of $CC_{dg}(A)$ starting at $CC_{dg}(A)^{0, p, q}$ is exact, except at first position. And so the only thing contributing to the cohomology of $\mc{T}(CC_{dg}(A))$ is 
\[
\begin{tikzcd}[sep = small]
\ddots &\tvdots & \ddots& \tvdots &\ddots&\tvdots & \\
& C^1_\Lam(A)^2 \ar{rr} \ar{u}\ar{ul}     &                                  \tvdots                                  & 0 \ar{rr}\ar{ul}   \ar{u} \ar[from=dd]         &                            \tvdots                                      & 0 \ar{u} \ar{r} \ar[from=dd]   \ar{ul} & \mathrlap{\cdots}\tvdots &\\
\ddots&                                      & C^0_\Lam(A)^2 \ar{ul} \ar[crossing over]{rr}   \ar{u}  &                              &0 \ar{ul} \ar[crossing over]{rr} \ar{u}                                      &                                   & 0  \ar{ul}  \ar{u}   \ar{r} & \cdots    \\
&C^1_\Lam(A)^1 \ar{rr} \ar{uu} \ar{ul} &                                                                                  & 0  \ar[from = dd] \ar{rr}            &                                    & 0  \ar{r}\ar[from=dd]           &       \cdots                                                     \\
\ddots &                                     & C^0_\Lam(A)^1 \ar[crossing over]{rr} \ar{ul} \ar[crossing over]{uu} &                                               & 0 \ar[crossing over]{rr} \ar{ul} \ar[crossing over]{uu} &                                   & 0 \ar{ul} \ar{uu}    \ar{r} & \cdots          \\
&C^1_\Lam(A)^0 \ar{uu}{-\del} \ar{rr}\ar{ul} &                                                                                                           &0 \ar{rr} &                                                                                  & 0 \ar{r}&                                                       \cdots     \\
 &                                             &C^0_\Lam(A)^0 \ar{ul}{b} \ar[crossing over, near start]{uu}{\del} \ar[swap]{rr}&                                               & 0 \ar{ul} \ar[crossing over]{uu} \ar[swap]{rr}  &       & 0 \ar{ul} \ar{uu}\ar{r} & \cdots \\
\end{tikzcd}
\]
It is clear then that $\mc{T}(CC_{dg}(A))$ is quasi-isomorphic to $\mc{T}(C_\Lam(A))$.
\end{proof}

Since $(T(C_{\Lam}(A)), d)$ is a subcomplex of $(\mc{T}(EH(A)), d)$ then as in the non-dga case we have a long exact sequence between $HC^*_{dg}$ and $HH^*_{dg}$
\[
\begin{tikzcd}[cramped, sep=small]
\cdots \ar{r} & HH^n_{dg}(A) \ar{r} & HC_{dg}^{n-1}(A) \ar{r} & HC_{dg}^{n+1}(A) \ar{r} & HH_{dg}^{n+1}(A) \ar{r} & \cdots
\end{tikzcd}
\]
arising from the short exact sequence of complexes
\[
\begin{tikzcd}
0 \ar{r} & \mc{T}(C_\Lam(A)) \ar{r}{i} &\mc{T}(EH(A)) \ar{r}{\pi} &\mc{T}(EH(A))/\mc{T}(C_\Lam(A))\ar{r} & 0
\end{tikzcd}
\]
Now, consider the vertical filtration $F$ on $\mc{T}(EH(A))$ that gives rise to the characteristic spectral sequence $\prescript{}{c}{E}$ for Hochschild cohomology defined by Neumann and Szymik~\cite{neumann2016spectral}. In this case then the characteristic spectral sequence has a second page given by
\[
	\prescript{}{c}{E^{s,t}_2} \simeq HH^s_{gr}(H^*(A)^{\otimes s}, \Sigma^tH^*(\Hom{k}{A}{k})) \Rightarrow HH^{s+t}_{dg}(A)
\]
\begin{thm}
For every dg $\Bbbk$-algebra $(A, \phi)$ where $k$ is a field of characteristc $0$ there is a convergent spectral sequence $E$ calculating $HC^*_{dg}(A)$ such that there is a morphism of spectral sequences $\map{f}{E}{\prescript{}{c}{E}}$.
\end{thm}

\begin{proof}
Applying the vertical filtration $F$ to $\mc{T}(C_\Lam(A))$ this clearly induces a convergent spectral sequence $E$ calculating $HC^*_{dg}(A)$ for the same reasons that it induces a convergent spectral sequence for $HH^*_{dg}$. That is, the filtration is canonically bounded and so by~\cite[Theorem 2.6]{McClearyJohn2001Augt} the associated spectral sequence converges. Furthermore since $C^n_{\Lam}(A)^m\subset EH(A)^{n,m}$ then we immediately have that $F^s\mc{T}(C_\Lam(A))^n\subset F^s\mc{T}(EH(A))^n$. It is clear as well that the inclusion map $\map{i}{ \mc{T}(C_\Lam(A))}{\mc{T}(EH(A))}$ satisfies $id=di$ and therefore by~\cite[Theorem 3.5]{McClearyJohn2001Augt} this induces a morphism of spectral sequences $\map{f}{E}{\prescript{}{c}{E}}$.
\end{proof}

\begin{rem}
The same argument can be used to show that there is a spectral sequence that calculates $HC^*_{dg}(A)$ such that there is a morphism to the forgetful spectral sequence $\prescript{}{f}{E}$ as defined in~\cite{neumann2016spectral}. The difference being that the horizontal filtration is used instead of the vertical filtration. 
\end{rem}

From~\cite[Theorem 4.3]{specseqpolyphim2025} we have $6$ convergent spectral sequences for $HC^*_{dg}(A)$ coming from applying the canonically bounded filtrations associated with the restrictions on the indices of $EC^{p,q,r}$, and for dg $\Bbbk$-algebras we have an additional $2$ sequences coming from $\mc{T}(C_\Lam(A))$. Therefore we have the following result.

\begin{thm}
For dgas $(A, \phi)$ over an arbitrary ring $\Bbbk$ there are at least $6$ convergent spectral sequences for calculating $HC^*_{dg}(A)$. If $\Bbbk$ is a field of characteristic $0$ then there are at least $8$.  
\end{thm}

Now, consider~\cite[Lemma 4.4]{specseqpolyphim2025}, then the filtration $F_1$ on $\mc{T}(EC)$ is the filtration
\[
	F^s_1\mc{T}(EC)^n  = \bigoplus_{\substack{p+q+r =n\\ p\ge s}}EC^{p,q,r} 
\]
and this gives rise to a convergent spectral sequence $\prescript{}{1}{E}$ that converges to $HC^*_{dg}(A)$. Then by the same lemma each column $\prescript{}{1}{E^{p,*}_0}$ is the total complex of
\[
\begin{tikzcd}[cramped, sep=small]
\tvdots                             &                                  & \tvdots                                   & \tvdots               & \tvdots                &  \tvdots &  \\   
\arrow{r}0 \arrow{u}      & \arrow{r}\cdots        & \arrow{r} 0  \arrow{u}          & C^0(A)^2 \ar{r}\ar{u} & C^1(A)^2 \ar{r}\ar{u} & C^2(A)^2 \ar{r}\ar{u} & \cdots \\
 \arrow{r}0 \arrow{u}    &\arrow{r}\cdots   &  \arrow{r} 0 \arrow{u}                 &C^0(A)^1 \ar{r}\ar{u} & C^1(A)^1 \ar{r} \ar{u}& C^2(A)^1 \ar{r}\ar{u} & \cdots  \\
 \arrow{r}0\arrow{u}         & \arrow{r} \cdots        &\arrow{r} 0 \arrow{u}       & C^0(A)^0 \ar{r}[swap]{b}\ar{u}{\del} & C^1(A)^0 \ar{r}[swap]{b}\ar{u}{-\del} & C^2(A)^0 \ar{r}\ar{u} & \cdots
\end{tikzcd}
\]
with the first $p$ columns consisting of only $0$'s. This is clearly the Hochschild bicomplex $EH(A)$ preceded by $p$ columns of $0$'s. Therefore the total complex is a sequence of $p$ $0$'s and then $\mc{T}(EH(A))$. This gives rise to an $\prescript{}{1}{E_1}$-page as follows
\[
\begin{tikzpicture}
\draw[->] (0.75, 1) -- (10.5, 1); %p axis
\draw[->] (1, 0.75) -- (1, 6.5); %q axis

\node (1) at (0.75, 6.4) {$q$};
\node (2) at (10.25, 0.75) {$p$};

%row 1
\node (11) at (0.9,1) {$\scriptstyle HH^0_{dg}(A)$};
\node (12) at (2.5,1) {$0$};
\node (13) at (4,1) {$0$};
\node (14) at (5.5,1) {$0$};
\node (15) at (7,1) {$0$};
\node (16) at (8.5,1) {$0$};
\node (17) at (10,1) {$0$};
%row 2
\node (21) at (0.9,2) {$\scriptstyle HH^1_{dg}(A)$};
\node (22) at (2.5,2) {$\scriptstyle HH^0_{dg}(A)$};
\node (23) at (4,2) {$0$};
\node (24) at (5.5,2) {$0$};
\node (25) at (7,2) {$0$};
\node (26) at (8.5,2) {$0$};
\node (27) at (10,2) {$0$};
%row 3
\node (31) at (0.9,3) {$\scriptstyle HH^2_{dg}(A)$};
\node (32) at (2.5,3) {$\scriptstyle HH^1_{dg}(A)$};
\node (33) at (4,3) {$\scriptstyle HH^0_{dg}(A)$};
\node (34) at (5.5,3) {$0$};
\node (35) at (7,3) {$0$};
\node (36) at (8.5,3) {$0$};
\node (37) at (10,3) {$0$};
%row 4
\node (41) at (0.9,4) {$\scriptstyle HH^3_{dg}(A)$};
\node (42) at (2.5,4) {$\scriptstyle HH^2_{dg}(A)$};
\node (43) at (4,4) {$\scriptstyle HH^1_{dg}(A)$};
\node (44) at (5.5,4) {$\scriptstyle HH^0_{dg}(A)$};
\node (45) at (7,4) {$0$};
\node (46) at (8.5,4) {$0$};
\node (47) at (10,4) {$0$};
%row 5
\node (51) at (0.9,5) {$\scriptstyle HH^4_{dg}(A)$};
\node (52) at (2.5,5) {$\scriptstyle HH^3_{dg}(A)$};
\node (53) at (4,5) {$\scriptstyle HH^2_{dg}(A)$};
\node (54) at (5.5,5) {$\scriptstyle HH^1_{dg}(A)$};
\node (55) at (7,5) {$\scriptstyle HH^0_{dg}(A)$};
\node (56) at (8.5,5) {$0$};
\node (57) at (10,5) {$0$};
%row 6
\node (61) at (0.9,6) {$\scriptstyle HH^5_{dg}(A)$};
\node (62) at (2.5,6) {$\scriptstyle HH^4_{dg}(A)$};
\node (63) at (4,6) {$\scriptstyle HH^3_{dg}(A)$};
\node (64) at (5.5,6) {$\scriptstyle HH^2_{dg}(A)$};
\node (65) at (7,6) {$\scriptstyle HH^1_{dg}(A)$};
\node (66) at (8.5,6) {$\scriptstyle HH^0_{dg}(A)$};
\node (67) at (10,6) {$0$};

\end{tikzpicture}
\]
Note that this immediately implies that $\prescript{}{1}{E_\infty^{p,q}}\simeq \{0\} $ for all $q<p$ and $\prescript{}{1}{E_2^{p,p}}\simeq HH^0_{dg}(A)/\im{HH^1_{dg}(A)\to HH^0_{dg}(A)}$ for all $p>0$. We also see that $HC^0_{dg}(A)\simeq\prescript{}{1}{E_\infty^{0,0}}\simeq HH^0_{dg}(A)$, as is the case when $A$ is a non-graded algebra. 

\begin{defn}
Let $V$ be a graded vector space. The \emph{$k$-th negative truncation} of $V$ is the graded vector space $\Delta_{-}^kV$ given by
\[
(\Delta_{-}^kV)^n =
	\begin{cases}
	V^{k-n}& n\le k \\
	0& n>k
	\end{cases}
\]
\end{defn}

Each row $\prescript{}{1}{E_1^{\ast, q}}\simeq \Delta_{-}^qHH^*_{dg}(A)$ and therefore the second page $\prescript{}{1}{E_2}$ is given by
\[
	\prescript{}{1}{E_2^{p,q}}\simeq H^p(\Delta_{-}^qHH^*_{dg}(A))
\]
And so we have:
\begin{thm}
\label{spec seq for hc with hh as E2}
For each dga $(A,\phi)$ there exists a spectral sequence with $E_2^{p,q}\simeq H^p(\Delta_{-}^qHH^*_{dg}(A))$ that converges to the cyclic cohomology $HC^{p+q}_{dg}(A)$.
\end{thm}

\begin{prop}
\label{degen at E2 gives surj}
If the spectral sequence in Theorem~\ref{spec seq for hc with hh as E2} degenerates at the $E_2$ page then there is a surjection $$HC^n_{dg}(A)\to\ker{(HH^n_{dg}(A)\to HH^{n-1}_{dg}(A))}$$ for all $n\ge 1$.
\end{prop}

\begin{proof}
From Theorem~\ref{spec seq for hc with hh as E2} then 
\[
	\prescript{}{1}{E_2^{p,q}}\simeq H^p(\Delta_{-}^qHH^*_{dg}(A))
\]
and therefore 
\[
	\prescript{}{1}{E^{0,n}_2}\simeq H^0(\Delta_{-}^nHH^*_{dg}(A))\simeq\ker{(HH^n_{dg}(A)\to HH^{n-1}_{dg}(A))}.
\]
The edge homomorphism $HC^n_{dg}(A)\to\prescript{}{1}{E^{0,n}_2}$ is the composition of the surjection $HC^n_{dg}(A)\to \prescript{}{1}{E^{0,n}_\infty}$, and the injection $\prescript{}{1}{E^{0,n}_\infty}\to\prescript{}{1}{E^{0,n}_2}$ coming from the tower of inclusions
\[
	\prescript{}{1}{E^{0,n}_\infty}\subseteq\cdots\subseteq\prescript{}{1}{E^{0,n}_2}.
\]
Then when the spectral sequence in Theorem~\ref{spec seq for hc with hh as E2} degenerates at the $E_2$ page then $\prescript{}{1}{E^{0,n}_\infty}=\prescript{}{1}{E^{0,n}_2}$ and therefore the edge homomorphism $HC^n_{dg}(A)\to\prescript{}{1}{E^{0,n}_2}$ is surjective.
\end{proof}
\begin{prop}
\label{hc1 and hc2}
For all dgas $(A, \phi)$ we have
\begin{align*}
	&HC^1_{dg}(A) \simeq H^0(\Delta_{-}^1HH^*_{dg}(A)) \\
	&HC^2_{dg}(A) \simeq H^0(\Delta_{-}^2HH^*_{dg}(A))\oplus H^1(\Delta_{-}^1HH^*_{dg}(A))
\end{align*}
\end{prop}
\begin{proof}
From inspection of the spectral sequence in Theorem~\ref{spec seq for hc with hh as E2} we see that the terms $E_2^{0,1}$, $E^{0,2}_2$ and $E_2^{1,1}$ stabilise, and therefore 
\begin{align*}
&E_\infty^{0,1}\simeq E_2^{0,1} \simeq H^0(\Delta_{-}^1HH^*_{dg}(A))\\
&E_\infty^{0,2}\simeq E_2^{0,2} \simeq H^0(\Delta_{-}^2HH^*_{dg}(A))\\
&E_\infty^{1,1}\simeq E_2^{1,1} \simeq H^1(\Delta_{-}^1HH^*_{dg}(A))
\end{align*}
and since 
\[
	HC^n(A) \simeq\bigoplus_{p+q = n}E_\infty^{p,q}
\]
and $E_\infty^{p,q} = \{0\}$ for $q<p$, then we have
\begin{align*}
	&HC^1_{dg}(A) \simeq E_\infty^{0,1}\\
	&HC^2_{dg}(A) \simeq E_\infty^{0,2}\oplus E_\infty^{1,1}
\end{align*}
and so the result follows.
\end{proof}
Now consider the filtration $F_2$ of $\mc{T}(EC)$ given by
\[
	F^s_2\mc{T}(EC)^n = \bigoplus_{\substack{p+q+r=n\\ q\ge s}}EC^{p,q,r}
\]
Then from the proof of~\cite[Lemma 4.4]{specseqpolyphim2025} each column $\prescript{}{2}{E_0^{p,\ast}}$ is the total complex of 
 \[
\begin{tikzcd}
\tvdots                                                                                        & \tvdots                                                                    &                                    & \tvdots                         & \tvdots          \\   
\arrow{r}C^p(A)^2 \arrow{u}                                                 & \arrow{r}C^{p-1}(A)^2\ar{u}                             & \arrow{r} \cdots        & C^0(A)^2 \ar{r}\ar{u} & 0 \ar{u} \\
 \arrow{r}C^p(A)^1 \arrow{u}                                               &\arrow{r}C^{p-1}(A)^1 \ar{u}                              &  \arrow{r} \cdots         &C^0(A)^1 \ar{r}\ar{u} & 0  \ar{u} \\
 \arrow{r}[swap]{B} C^p(A)^0\arrow{u}{(-1)^p\del}         & \arrow{r} C^{p-1}(A)^0 \ar{u}[swap]{(-1)^{p-1}\del} &\arrow{r} \cdots         & C^0(A)^0 \ar{r}\ar{u}{\del} & 0\ar{u}
\end{tikzcd}
\]
And therefore $\prescript{}{2}{E^{p,q}_1}$ is given by
\[
	\prescript{}{2}{E^{p,q}_1}\simeq H^q\left(\bigoplus_{i=0}^p\Sigma^{-i}C^{p-i}(A)^\ast\right)
\]
where $\Sigma^{-i}C^{k}(A)^n = C^{k}(A)^{n-i}$. Unfortunately we cannot commute the direct sum and cohomology since the differential on $\prescript{}{2}{E_0^{p,\ast}}$ doesn't decompose as a direct sum along the same pattern as the direct sum decomposition of $\prescript{}{2}{E^{p,q}_1}$. So each row $\prescript{}{2}{E^{\ast, q}_2}$ is the cochain complex
\[
\begin{tikzcd}[cramped, sep=small]
H^q(C^0(A)^\ast) \ar{r} &\cdots \ar{r} &  H^q\left(\bigoplus_{i=0}^p\Sigma^{-i}C^{p-i}(A)^\ast\right) \ar{r} & \cdots
\end{tikzcd}
\]
Then denoting this cochain complex by $H^q\left(\bigoplus\Sigma C(A)^\ast\right)$ then $\prescript{}{2}{E^{p,q}_2}$ is given by
\[
	\prescript{}{2}{E^{p,q}_2}\simeq H^p\left(H^q\left(\bigoplus\Sigma C(A)^\ast\right)\right)
\]
And so we have the following result.
\begin{thm}
\label{spec seq for hc from f2 filt}
For each dga $(A,\phi)$ there's a spectral sequence with $E_2^{p,q}\simeq H^p\left(H^q\left(\bigoplus\Sigma C(A)^\ast\right)\right)$ that converges to the cyclic cohomology $HC^{p+q}_{dg}(A)$.
\end{thm}

For the next spectral sequence that we will discuss, we introduce the following variation on cyclic cohomology for dgas.
\begin{defn}
\label{partial cyclic cohom}
Let $(A,\phi)$ be a dga, then the \emph{$n$-th partial cyclic cohomology $HCP^\ast_{n}(A)$} of $A$ is the cohomology of the total complex of 
\[
\begin{tikzcd}
\tvdots &\tvdots &\tvdots   \\   
\arrow{u} C^2(A)^n \arrow{r}{B}&\arrow{u} C^1(A)^n \arrow{r}{B}&\arrow{u} C^0(A)^n \\
\arrow{u}{b} C^1(A)^n  \arrow{r}{B}& \arrow{u}{b} C^0(A)^n    \\
\arrow{u}{b} C^0(A)^n 
\end{tikzcd}
\]
where we denote the bicomplex above by $\mathcal{BP}^n(A)$. Therefore 
\[
	HCP^m_n(A) = H^m(\mc{T}(\mathcal{BP}^n(A)))
\]
for each $m\ge 0$. 
\end{defn}

\begin{prop}
\label{0th hcp}
$HCP^\ast_0(A)\simeq HC^\ast(A^0)$
\end{prop}

\begin{proof}
$\mathcal{BP}^0(A)$ is the bicomplex
\[
\begin{tikzcd}
\tvdots &\tvdots &\tvdots   \\   
\arrow{u} C^2(A)^0 \arrow{r}{B}&\arrow{u} C^1(A)^0 \arrow{r}{B}&\arrow{u} C^0(A)^0 \\
\arrow{u}{b} C^1(A)^0  \arrow{r}{B}& \arrow{u}{b} C^0(A)^0    \\
\arrow{u}{b} C^0(A)^0 
\end{tikzcd}
\]
and $C^m(A)^0 = \Hom{\Bbbk}{(A^{\otimes m+1})^0}{\Bbbk}$. However $(A^{\otimes m+1})^0 = (A^0)^{\otimes m+1}$, and therefore $\mathcal{BP}^0(A) = \mathcal{B}(A^0)$. Since $A^0$ inherits an algebra structure from the algebra structure of $A$, then taking the cohomology of the total complex of $\mathcal{B}(A^0)$ gives us the cyclic cohomology of $A^0$, and so $HCP^\ast_0(A)\simeq HC^\ast(A^0)$.
\end{proof}

\begin{prop}
\label{partial cyc cohom is subset}
$HCP^0_n(A)\subset HC^n_{dg}(A)$ for all $n$.
\end{prop}

\begin{proof}
Consider
\[
\begin{tikzcd}[cramped, sep=small]
  C^0(A)^0\arrow{r}& \cdots \ar{r} &\mc{T}(EC)^{n-1}\ar{r}& \mc{T}(EC)^n\ar{r} & \mc{T}(EC)^{n+1}\ar{r} & \cdots \\
   0\ar{r}& \cdots \ar{r}& 0\ar{r}& C^0(A)^n \ar{r}[swap]{b} \ar{u}{i}& \ar{r} C^1(A)^n\ar{r}\ar{u}{i} & \cdots
\end{tikzcd}
\]
Then $\map{i}{\Sigma^{-n}\mc{T}(\mathcal{BP}^n(A))}{\mc{T}(EC)}$ is an injective cochain map. Clearly, $HCP^0_n(A) = \ker{(b)}$ and so each cohomology class has only one representative, that is, the coboundaries are $\{0\}$. Now take $f\ne g\in\ker{(b)}$ and suppose that
\[
	i^*[f] = [i(f)] = [i(g)]  =i^*[g]
\]
This is true if and only if $i(f)-i(g)\in B^n(\mc{T}(EC))$ and therefore $i(f-g)\in B^n(\mc{T}(EC))$. But since $i$ is a cochain map then it must send coboundaries to coboundaries, and so this implies that $f-g$ is a coboundary in $C^0(A)^n$. But this implies $f = g$ which is a contradiction. Therefore $i^*[f]\ne i^*[g]$ and thus $\map{i^*}{HCP^0_n(A)}{HC^n_{dg}(A)}$ is injective.
\end{proof}

Now, consider the filtration $F_3$ given by
\[
	F_3^s\mc{T}(EC)^n = \bigoplus_{\substack{p+q+r=n\\ r\ge s}}EC^{p,q,r}
\]
Then by~\cite[Lemma 4.4]{specseqpolyphim2025} each column $\prescript{}{3}{E_0^{p,\ast}}$ is the total complex of 
\[
\begin{tikzcd}
\tvdots &\tvdots &\tvdots   \\   
\arrow{u} C^2(A)^p \arrow{r}{B}&\arrow{u} C^1(A)^p \arrow{r}{B}&\arrow{u} C^0(A)^p \\
\arrow{u}{b} C^1(A)^p  \arrow{r}{B}& \arrow{u}{b} C^0(A)^p    \\
\arrow{u}{b} C^0(A)^p 
\end{tikzcd}
\]
and therefore $\prescript{}{3}{E_1^{p,\ast}} \simeq HCP^\ast_p(A)$. In particular $\prescript{}{3}{E_1^{0,\ast}} \simeq HC^*(A^0)$. And so the rows $\prescript{}{3}{E_1^{\ast, q}}$ are given by
\[
	\prescript{}{3}{E_1^{\ast, q}} \simeq HCP^q_\ast(A)
\]
and therefore the second page $\prescript{}{3}{E_2}$ is given by
\[
	\prescript{}{3}{E_2^{p,q}} \simeq H^p(HCP^q_\ast(A))
\]
Thus we have:
\begin{thm}
\label{spec seq hcp}
For every dga $(A,\phi)$ there exists a spectral sequence with $E_2^{p,q}\simeq H^p(HCP^q_\ast(A))$ that converges to the cyclic cohomology $HC^{p+q}_{dg}(A)$.
\end{thm}

The final spectral sequence that we will construct comes from the filtration $F_{1,3}$ given by
\[
	F^s_{1,3}\mc{T}(EC)^n = \bigoplus_{\substack{p+q+r=n\\ p+r\ge s}}EC^{p,q,r}
\]
and for this spectral sequence we introduce a variation of Hochschild cohomology of a dga, analogous to Definition~\ref{partial cyclic cohom}.
\begin{defn}
Let $(A,\phi)$ be a dga, then the \emph{$n$-th partial Hochschild cohomology} $HHP^\ast_n(A, M)$ \emph{with coefficients in the dg A-bimodule $M$} is the cohomology of
\[
\begin{tikzcd}
M^n \arrow{r}{b} & \Hom{\Bbbk}{A}{M}^n \arrow{r}{b} &  \Hom{\Bbbk}{A^{\otimes 2}}{M}^n \arrow{r} & \cdots
\end{tikzcd}
\]
This is the \emph{$n$-th partial Hochschild complex} $CP^n(A, M)$ and so 
\[
	HHP^m_n(A, M) = H^m(CP^n(A, M))
\]
for each $m\ge 0$. When $M = A^\ast$ then we write $HHP^\ast_n(A) := HHP^\ast_n(A, A^\ast)$.
\end{defn}

\begin{prop}
$HHP^\ast_0(A, M)\simeq HH^*(A^0, M)$
\end{prop}

\begin{proof}
The proof is exactly analogous to the proof for Proposition~\ref{0th hcp}. 
\end{proof}

\begin{prop}
$HHP^0_n(A)\subset HH^n_{dg}(A)$ for all $n$.
\end{prop}

\begin{proof}
Again the proof is analogous to the proof of Proposition~\ref{partial cyc cohom is subset}.
\end{proof}
In particular $HHP^\ast_0(A) \simeq HH^\ast(A^0)$. By~\cite[Lemma 4.4]{specseqpolyphim2025} the column $\prescript{}{1,3}{E_0^{0,\ast}}$ is the cochain complex 
\[
\begin{tikzcd}
\tvdots                                            \\              
 C^2(A)^0\arrow{u}{b}             \\                   
 C^1(A)^0 \arrow{u}{b}                          \\     
 C^0(A)^0\arrow{u}{b}                  
\end{tikzcd}
\]
which is the Hochschild complex for $A^0$, and therefore $\prescript{}{1,3}{E_1^{0,\ast}}\simeq HH^\ast(A^0)$. The column $\prescript{}{1,3}{E_0^{1,\ast}}$ is the total complex of
\[
\begin{tikzcd}
\tvdots                                        & \tvdots                          \\   
 C^2(A)^1 \arrow{u}{b}                 & C^1(A)^0 \arrow{u}{b}        \\
 C^1(A)^1 \arrow{u}{b}                 & C^0(A)^0 \arrow{u}{b}    \\
 C^0(A)^1\arrow{u}{b}             & 
\end{tikzcd}
\]
Which is the cochain complex $C^*(A)^1\oplus \Sigma^{-1}C^*(A)^0$. The differential also decomposes as a sum $b + b$ along the same pattern, so the cohomology splits as a direct sum too, and therefore 
\[
	\prescript{}{1,3}{E_1^{1,\ast}}\simeq HHP^\ast_1(A)\oplus \Sigma^{-1}HH^*(A^0).
\]
Continuing in this way we see that in general $\prescript{}{1,3}{E_1^{p,\ast}}$ is given by
\[
	\prescript{}{1,3}{E_1^{p,\ast}}\simeq \bigoplus_{i=0}^p\Sigma^{-i}HHP^*_{p-i}(A).
\]
Now, each row $\prescript{}{1,3}{E_1^{\ast, q}}$ is the cochain complex
\[ 
\begin{tikzcd}[cramped, sep=small]
HH^q \ar{r} & HHP^q_1\oplus HH^{q-1} \ar{r}&HHP^q_2\oplus HHP^{q-1}_1\oplus HH^{q-2}\ar{r} & \cdots&.
\end{tikzcd}
\]
Where we've abbreviated symbols for the purposes of readability. We will denote this cochain complex by $\bigoplus HHP^q_\ast(A)$ and therefore 
\[
	\prescript{}{1,3}{E_2^{p,q}}\simeq H^p(\bigoplus HHP^q_\ast(A)).
\]
And thus we have:

\begin{thm}
\label{spec seq hhp}
For every dga $(A,\phi)$ there exists a spectral sequence with $E_2^{p,q}\simeq H^p(\bigoplus HHP^q_\ast(A))$ that converges to the cyclic cohomology $HC^{p+q}_{dg}(A)$.
\end{thm}

We don't explore the spectral sequences coming from the filtrations $F_{1,2}$ and $F_{2,3}$ because it involves finding the cohomology of the cochain complexes 
\[
\begin{tikzcd}
C^n(A)^0\ar{r}{\del} & C^n(A)^1\ar{r}{\del}& C^n(A)^2 \ar{r}{\del} & \dots
\end{tikzcd}
\]
and
\[
\begin{tikzcd}
C^n(A)^m \ar{r}{B} & C^{n-1}(A)^m \ar{r}{B} & \dots \ar{r}{B}& C^0(A)^m
\end{tikzcd}
\]
respectively, which are less well understood. However, it may be interesting to explore this further. 

\begin{prop}
\label{spec seq morph 1 to 13 and 3 to 13}
Let $\prescript{}{1}{E}$, $\prescript{}{3}{E}$ and $\prescript{}{1,3}{E}$ be the spectral sequences from Theorem~\ref{spec seq for hc with hh as E2}, Theorem~\ref{spec seq hcp} and Theorem~\ref{spec seq hhp} respectively. Then there are morphisms of spectral sequences $\map{f}{\prescript{}{1}{E}}{\prescript{}{1,3}{E}}$ and $\map{g}{\prescript{}{3}{E}}{\prescript{}{1,3}{E}}$.
\end{prop}

\begin{proof}
Since $\prescript{}{1,3}{E}$ comes from the condition $p+r\ge s$ and clearly $p, r\in\{p,r\}$ then this is a direct application of~\cite[Proposition 4.5]{specseqpolyphim2025}. 
\end{proof}

\begin{corr}
For every dga $(A,\phi)$ there are maps 
\begin{align*}
	& HH^n_{dg}(A)\to HH^n(A^0) \\
	& HCP^0_n(A)\to HHP^0_n(A) \\
	& HC^n(A^0)\to HH^n(A^0)
\end{align*}
for all $n\ge 0$.
\end{corr}

\begin{proof}
From Proposition~\ref{spec seq morph 1 to 13 and 3 to 13} we know that there are morphisms of spectral sequences $\map{f}{\prescript{}{1}{E}}{\prescript{}{1,3}{E}}$ and $\map{g}{\prescript{}{3}{E}}{\prescript{}{1,3}{E}}$. Therefore for every $p,q$ and $r$ there are maps 
\begin{align*}
	&\map{f^{p,q}_r}{\prescript{}{1}{E^{p,q}_r}}{\prescript{}{1,3}{E^{p,q}_r}} \\
	&\map{g^{p,q}_r}{\prescript{}{3}{E^{p,q}_r}}{\prescript{}{1,3}{E^{p,q}_r}} 
\end{align*}
Then consider the $E_1$ pages of $\prescript{}{1}{E}$, $\prescript{}{3}{E}$ and $\prescript{}{1,3}{E}$. Then all maps can be found by comparing $E_1$ pages. For instance, the maps $HH^n_{dg}(A)\to HH^n(A^0)$ are the maps $\map{f^{0,n}_1}{\prescript{}{1}{E^{0,n}_1}}{\prescript{}{1,3}{E^{0,n}_1}}$.
\end{proof}

\begin{rem}
It's not clear if the maps $ HC^n(A^0)\to HH^n(A^0)$ in the corollary above are the same maps that relate cyclic and Hochschild cohomology in Connes' long exact sequence. It may be interesting to examine if there is a difference between these two maps. 
\end{rem}

\section{Spectral Sequences for DG-Categories}
\label{spec seq of dgcat}
Up until this point we have used the notation $\Hom{\mc{A}}{X}{Y}$ to denote the hom set between objects $X$ and $Y$ in the category $\mc{A}$. However, in this section we will use $\mc{A}(X, Y)$ instead.

Let $\mc{A}$ be a small dg-category and $M$ an $\mc{A}$-bimodule. We define the cosimplicial module $EH(\mc{A}, M)^n$ as
\[
	\prod_{(X_0,\dots, X_n)}\Hom{\Bbbk}{\mc{A}(X_{n-1}, X_n)\otimes\mc{A}(X_{n-2}, X_{n-1})\otimes\cdots\otimes\mc{A}(X_0, X_1)}{M(X_0, X_n)}
\]
where $(X_0, \dots, X_n)$ range over the objects of $\mc{A}$, that is, the product ranges over all possible sequences of $n+1$ objects in $\mc{A}$. The $0$-th part of this cosimplicial module is 
\[
	EH(\mc{A}, M)^0 = \prod_{X_0} M(X_0, X_0)
\]
which ranges over every object in $\mc{A}$.

We see then why it is necessary for $\mc{A}$ to be a small dg-category, otherwise it wouldn't be possible to form these products. 

Each factor of $EH^n(\mc{A}, M)$ is a cosimplicial module in its own right with the familiar coface maps $\partial^i$ and codegenerancy maps $\sigma^i$ that we have seen already for associative algebras and dg-algebras. $EH(\mc{A}, M)$ is then the product cosimplicial module of these factors. 

Each factor of $EH(\mc{A}, M)^n$ is also a dg-module and has the natural grading
\[
	\Hom{\Bbbk}{\mc{A}(X_{n-1}, X_n)\otimes\cdots\otimes\mc{A}(X_0, X_1)}{M(X_0, X_n)}^m.
\]
Therefore $EH(\mc{A}, M)^n$ is a dg-module, where 
\[
	 EH(\mc{A}, M)^{n,m} = \prod_{(X_0,\dots, X_n)}\Hom{\Bbbk}{\mc{A}(X_{n-1}, X_n)\otimes\cdots\otimes\mc{A}(X_0, X_1)}{M(X_0, X_n)}^m
\]
and so it follows that $EH(\mc{A}, M)$ is a bicomplex.

\begin{defn}
The \emph{Hochschild complex $C_{dg}(\mc{A})$} of the dg-category $\mc{A}$ is the product total complex $\mc{T}(EH(\mc{A}, M))$. The \emph{Hochschild cohomology $HH^\ast_{dg}(\mc{A}, M)$} of the dg-category $\mc{A}$ with coefficients in $M$ is then defined as $HH^n_{dg}(\mc{A}, M) = H^n(C_{dg}(\mc{A}))$.
\end{defn}

Compare with \cite[Section 6]{neumann2016spectral}, and \cite[Section 5.4]{keller2006differential}. The spectral sequences that we discussed for Hochschild cohomology of dgas extend naturally to the dg-category setting. In fact as we see from~\cite[Theorem 7.1]{neumann2016spectral} this is how it was first described, and we have explored the special case for a one object dg-category, i.e. a dga. So we have the following:

\begin{thm}
For every dg-category $\mc{A}$ and every $\mc{A}$-bimodule $M$ there is a spectral sequence $E_2^{p,q}\simeq HH^p_{gr}(H_\bullet\mc{A}, \Sigma^tH_\bullet M)\Rightarrow HH^{p+q}_{dg}(\mc{A}, M)$, where $H_\bullet\mc{A}$ is the homology category of $\mc{A}$ and $\Sigma^tH_\bullet M$ is the shifted homology of $M$. 
\end{thm}

Now, for cyclic cohomology the literature focuses mostly on cyclic \emph{homology} of dg-categories (and variations). For instance see~\cite{KellerBernhard2023Aitr} and~\cite{ShklyarovDmytro2017Csoc} for a definition of the cyclic homology of dg-categories and how it relates to Calabi-Yau structures. A definition of the periodic cyclic homology of $\mathbb{Z}/2$-graded dg-categories is given in~\cite{EfimovAlexanderI2018CHoC}. We expect the cyclic cohomology of dgas to extend naturally to dg-categories, just as it does with cyclic homology, and also the spectral sequences we have described for its calculation.

For classical cyclic cohomology and cyclic cohomology of dgas, we take the coefficients to be the dual algebra $A^\ast=\Hom{\Bbbk}{A}{\Bbbk}$. In the dg-category setting we take the coefficients $M$ to be the $\mc{A}$-bimodule $\map{\mathcal{D}}{\mc{A}^{op}\otimes\mc{A}}{\textbf{dgMod}_\Bbbk}$ that sends each pair of objects $X_1, X_2\in\mc{A}$ to the dual of their hom set, that is
\[
	\mathcal{D}(X_1, X_2) = \mc{A}(X_1, X_2)^\ast
\]
Since dg-categories are rigid~\cite{Bertrand2011}, that is all objects have a dual object, and the category $\textbf{dgMod}_\Bbbk$ of differential graded $\Bbbk$-vector spaces is a dg-category then $\mathcal{D}$ is a well-defined functor. We set $EH(\mc{A}) := EH(\mc{A}, \mc{D})$ and so
\[
	EH(\mc{A})^{n,m} = \prod_{(X_0,\dots, X_n)}\Hom{\Bbbk}{\mc{A}(X_{n-1}, X_n)\otimes\cdots\otimes\mc{A}(X_0, X_1)}{\mc{D}(X_0, X_n)}^m
\]
$EH(\mc{A})$ is a cosimplicial module, and with the cyclic action it is a cocyclic module, and therefore $EH(\mc{A})$ make up the components of the cyclic bicomplex which expands as the cyclic tricomplex $EC(\mc{A})$ analagous to the cyclic tricomplex in Definition~\ref{hc defn with cyc tricomp}. Therefore we have the following definition.

\begin{defn}
The \emph{cyclic cohomology $HC^\ast_{dg}(\mc{A})$} of the dg-category $\mc{A}$ is the cohomology of the total complex of $EC(\mc{A})$.
\end{defn}

Since $EC(\mc{A})$ is a tricomplex then by~\cite[Theorem 4.3]{specseqpolyphim2025} there are at least $6$ spectral sequences that calculate $HC^*_{dg}(\mc{A})$, and in particular they are the same spectral sequences as the ones that calculate $HC^*_{dg}(A)$ when $A$ is a dga. Consider the spectral sequence from Theorem~\ref{spec seq for hc with hh as E2}, then in exactly the same way each column $\prescript{}{1}{E^{p,\ast}_0}$ is the Hochschild bicomplex $EH(\mc{A})$ preceded by $p$ columns of $0$'s. Therefore following the same reasoning we get:

\begin{thm}
\label{spec seq for hc of dgcat}
For every dg-category $\mc{A}$ there's a spectral sequence with $E_2^{p,q}\simeq H^p(\Delta^q_{-}HH^*_{dg}(\mc{A}))$ that converges to the cyclic cohomology $HC^{p+q}_{dg}(\mc{A})$, where $HH^*_{dg}(\mc{A}):=HH^*_{dg}(\mc{A}, \mc{D})$. 
\end{thm}

Furthermore, the results Proposition~\ref{degen at E2 gives surj} and Proposition~\ref{hc1 and hc2} also apply here using the same reasoning. Therefore we have the following: 

\begin{prop}
For all dg-categories $\mc{A}$ we have
\begin{align*}
	& HC^0_{dg}(\mc{A})\simeq HH^0_{dg}(\mc{A})\\
	&HC^1_{dg}(\mc{A}) \simeq H^0(\Delta_{-}^1HH^*_{dg}(\mc{A})) \\
	&HC^2_{dg}(\mc{A}) \simeq H^0(\Delta_{-}^2HH^*_{dg}(\mc{A}))\oplus H^1(\Delta_{-}^1HH^*_{dg}(\mc{A}))
\end{align*}
\end{prop}

\begin{prop}
When the spectral sequence from Theorem~\ref{spec seq for hc of dgcat} degenerates at the $E_2$ page then there is a surjection
\[
	HC^n_{dg}(\mc{A})\to \ker{(HH^n_{dg}(\mc{A})\to HH^{n-1}_{dg}(\mc{A}))}
\]
for all $n\ge 1$.
\end{prop}
\bibliographystyle{abbrv}
\bibliography{spec_seq_for_cyc_refs}
\end{document}